\documentclass[11pt]{amsart}
\usepackage{geometry}                
\usepackage{graphicx}
\usepackage{amssymb}
\usepackage{epstopdf}
\newtheorem{theorem}{Theorem}[section]
\newtheorem{lemma}{Lemma}[section]
\newtheorem{corollary}{Corollary}[section]

\newtheorem{definition}{Definition}[section]

\numberwithin{equation}{section}

\def\Z{\Bbb Z}

\def\R{\Bbb R}

\def\C{\Bbb C}

\def\N{\Bbb N}
\def\H{\Bbb H}

\def\D{\Delta}
\def\d{\partial}
\def\s{\sigma}
\def\e{\epsilon}
\def\a{\alpha}
\def\b{\beta}
\def\g{\gamma}

\title{The Stratified Spaces of Real Polynomials  \&  Trajectory Spaces of Traversing Flows}
\author{Gabriel Katz}
\address{5 Bridle Path Circle, Framingham, MA 01701, USA}
\email{gabkatz@gmail.com}

\begin{document}

\maketitle

\begin{abstract} 
This paper is the third in a series that researches the Morse Theory, gradient flows, concavity and complexity on smooth compact manifolds with boundary. Employing the local analytic models from \cite{K2}, for \emph{traversally generic flows} on $(n+1)$-manifolds $X$, we embark on a detailed and somewhat tedious study of universal combinatorics of their tangency patterns with respect to the boundary $\d X$. This combinatorics is captured by a universal poset $\Omega^\bullet_{'\langle n]}$ which depends only on the dimension of $X$.   It is intimately linked with the combinatorial patterns of real divisors of real polynomials in one variable of degrees which do not exceed $2(n+1)$.  Such patterns are elements of another natural poset $\Omega_{\langle 2n+2]}$ that describes the ways in which the real roots merge, divide, appear, and disappear under deformations of real polynomials.  The space of real degree $d$ polynomials $\mathcal P^d$ is stratified so that its pure strata are cells, labelled by the elements of the poset $\Omega_{\langle d]}$.  This cellular structure in $\mathcal P^d$  is interesting on its own right (see Theorem \ref{th4.1} and Theorem \ref{th4.2}). Moreover,  it helps to understand the \emph{localized} structure of the trajectory spaces $\mathcal T(v)$ for traversally generic fields $v$, the main subject of Theorem \ref{th5.2} and Theorem \ref{th5.3}. 
\end{abstract}

\section{Basics Facts about Boundary Generic and Traversally Generic Vector Fields}

For the reader convenience, we start with  a short review of  few key definitions and lemmas from \cite{K1} and \cite{K2}.
\smallskip

Let $v$ be a vector field on a smooth compact $(n+1)$-manifold $X$ with boundary $\d X$. To achieve some uniformity in our notations, let $\d_0X := X$ and $\d_1X := \d X$.

The vector field $v$ gives rise to a partition $\d_1^+X \cup \d_1^-X $ of the boundary $\d_1X$ into  two sets: the locus $\d_1^+X$, where the field is directed inward of $X$, and  $\d_1^-X$, where it is directed outwards. We assume that $v$, viewed as a section of the quotient  line bundle $T(X)/T(\d X)$ over $\d X$, is transversal to its zero section. This assumption implies that both sets $\d^\pm_1 X$ are compact manifolds which share a common boundary $\d_2X := \d(\d_1^+X) = \d(\d_1^-X)$. Evidently, $\d_2X$ is the locus where $v$ is \emph{tangent} to the boundary $\d_1X$.

Morse has noticed that, for a generic vector field $v$, the tangent locus $\d_2X$ inherits a similar structure in connection to $\d_1^+X$, as $\d _1X$ has in connection to $X$ (see \cite{Mo}). That is, $v$ gives rise to a partition $\d_2^+X \cup \d_2^-X $ of  $\d_2X $ into  two sets: the locus $\d_2^+X$, where the field is directed inward of $\d_1^+X$, and  $\d_2^-X$, where it is directed outward of $\d_1^+X$. Again, let us assume that $v$, viewed as a section of the quotient  line bundle $T(\d_1X)/T(\d_2X)$ over $\d_2X$, is transversal to its zero section.

For generic fields, this structure replicates itself: the cuspidal locus $\d_3X$ is defined as the locus where $v$ is tangent to $\d_2X$; $\d_3X$ is divided into two manifolds, $\d_3^+X$ and $\d_3^-X$. In  $\d_3^+X$, the field is directed inward of $\d_2^+X$, in  $\d_3^-X$, outward of $\d_2^+X$. We can repeat this construction until we reach the zero-dimensional stratum $\d_{n+1}X = \d_{n+1}^+X \cup  \d_{n+1}^-X$. 

These considerations motivate 

\begin{definition}\label{def1.1} 
We say that a smooth field $v$ on $X$ is \emph{boundary generic} if:
\begin{itemize}
\item $v|_{\d X} \neq 0$,
\item $v$, viewed as a section of the tangent bundle $T(X)$, is transversal to its zero section,
\item  for each $j = 1, \dots,  n+1$, the $v$-generated stratum $\d_jX$ is a  smooth submanifold of  $\d_{j-1}X$,
\item  the field  $v$, viewed as section of the quotient 1-bundle  $$T_j^\nu := T(\d_{j-1}X)/ T(\d_jX) \to \d_jX,$$ is transversal to the zero section of $T_j^\nu$ for all $j > 0$. 
\end{itemize}
We denote by $\mathcal V^\dagger(X)$ the space of all boundary generic fields on $X$.
\hfill\qed
\end{definition}

Thus a boundary generic vector field $v$  on $X$  gives rise to two Morse stratifications: 
\begin{eqnarray}\label{eq1.1}
\d X := \d_1X \supset \d_2X \supset \dots \supset \d_{n +1}X, \nonumber \\ 
X := \d_0^+ X \supset \d_1^+X \supset \d_2^+X \supset \dots \supset \d_{n +1}^+X 
\end{eqnarray}
, the first one by closed submanifolds, the second one---by compact ones.  Here $\dim(\d_jX) = \dim(\d_j^+X) = n +1 - j$. For simplicity, the notations ``$\d_j^\pm X$" do not reflect the dependence of these strata on the vector field $v$. When the field varies, we use a more accurate notation ``$\d_j^\pm X(v)$".
\smallskip

When $v$ is nonsingular on $\d_1X$, we can extend it into a larger manifold $\hat X$ so that $\hat X$ properly contains $X$ and the extension $\hat v$ remains nonsingular in in the vicinity of $\d_1X \subset \hat X$. Throughout this text, we treat the pair $(\hat X, \hat v)$ as a germ which extends $(X, v)$.
\smallskip

At each point $x \in \d_1X$, the $(-\hat v)$-flow defines the germ of the projection $p_x: \hat X \to S_x$, where $S_x$ is a local section of the $\hat v$-flow which is transversal to it.  The projection is considered at each point of  $\d_1X \subset \hat X$. When $\hat v$ is a gradient-like field for a function $\hat f: \hat X \to \R$, we can choose  the germ of the hypersurface $f^{-1}(f(x))$ for the role of $S_x$.
\smallskip

For boundary generic vector field $v$, we associate an ordered sequence of  multiplicities with each trajectory $\g$, such that $\g \cap \d_1X$ is a finite set.  In fact, for any $v \in \mathcal V^\dagger(X)$, the intersection $\{\a_i\} := \g \cap \d_1X$ is automatically a finite set. For traversing (see Definition 4.6 from \cite{K1}) generic fields, the points $\{a_i\}$ of the intersection $\g \cap \d_1X$  are ordered by the field-oriented trajectory $\g$, and the index $i$ reflects this ordering.

\begin{definition}\label{def1.2} Let $v \in \mathcal V^\dagger(X)$ be a generic field. Let $\g$ be a  $v$-trajectory which intersects the boundary $\d_1X$ at a finite number of points $\{a_i\}$. Each point $a_i$ belongs to a unique pure stratum $\d_{j_i}X^\circ$. 

The \emph{multiplicity} $m(\g)$ of $\g$ is defined by the formula
\begin{eqnarray}\label{eq1.2}
m(\g) = \sum_i\, j_i
\end{eqnarray}
The \emph{reduced multiplicity} $m'(\g)$ of $\g$ is defined by the formula
\begin{eqnarray}\label{eq1.3}
m'(\g) = \sum_i (j_i -1)
\end{eqnarray}
, and the \emph{virtual multiplicity} $\mu(\g)$ of $\g$ is defined by 
\begin{eqnarray}\label{eq1.4}
\mu(\g) = \sum_i \Big\lceil\frac{j_i}{2}\Big\rceil
\end{eqnarray}
, where $\lceil\sim\rceil$ denotes the integral part function. \hfill\qed
\end{definition}

For an open and dense subspace $\mathcal V^\ddagger(X)$ of $\mathcal V^\dagger(X)$, one can interpret $\mu(\g)$ as the maximal number of tangency points that any trajectory $\g'$ in the vicinity of $\g$ may have (see Theorem 3.4 from \cite{K2}).
\smallskip

When $v \in \mathcal V^\dagger(X)$, each set $\d_jX(v)$ is a manifold.
\smallskip 

Let $\d_jX(v)^\circ := \d_jX(v) \setminus \d_{j+1}X(v)$ denotes the pure Morse stratum.
\smallskip 

The following lemma (see Lemma 3.1 from \cite{K2} and \cite{Morin}) provides us with an analytic description of the Morse strata.

\begin{lemma}\label{lem1.1} Assume that $v$ is a boundary generic field. Denote by $\g_a$ the $\hat v$-trajectory through a point $a \in X$. If $a \in \d_kX(v)^\circ$
then in the vicinity of $a$ in $\hat X$, there exists a coordinate system $(u, x)$, where $u \in\R$ and $x \in \R^n$,  so that: 
\begin{itemize}
\item each $v$-trajectory  $\g$ is defined by an equation $\{x = \vec{const}\}$,
\item the boundary $\d_1X$ is defined by the equation
\begin{eqnarray}\label{eq1.5} 
u^k + \sum_{j =0}^{k -2} \; x_j\, u^j = 0
\end{eqnarray} 
\item each $v$-trajectory $\g$ hits only some strata $\{\d_jX(v)^\circ\}_{j \in J(a)}$ in such a way that  $$\sum_{j \in J(a)} j \leq k, \quad \text{and}\quad \sum_{j \in J(a)} j \equiv k\; (2).$$ \hfill\qed
\end{itemize} 
\end{lemma} 

The introduction of the next class $\mathcal V^\ddagger(X)$ of vector fields is inspired by the singularity theory of so called \emph{Boardman maps} with normal crossings (see \cite{Bo}, \cite{GG}).
\smallskip

Consider the collection of tangent spaces $\{T_{a_i}(\d_{j_i}X^\circ)\}_i$ to the pure strata $\{\d_{j_i}X^\circ\}_i$ that have a non-empty intersection with a given trajectory $\g$. 
 By Lemma \ref{lem1.1}, each space $T_{a_i}(\d_{j_i}X^\circ)$ is transversal to the curve $\g$. 
 \smallskip 

Let $S$ be a local section of the $\hat v$-flow at some point $a_\star \in \g$ and let $\mathsf T_\star$ be the  space tangent to $S$ at $a_\star$. Each space $T_{a_i}(\d_jX^\circ)$, with the help of the $\hat v$-flow, determines a vector subspace $\mathsf T_i = \mathsf T_i(\g)$ in  $\mathsf T_\star$. It is the image of  the tangent space $T_{a_i}(\d_jX^\circ)$ under the composition of two maps: (1) the differential of the flow-generated diffeomorphism that maps $a_i$ to $a_\star$ and (2) the linear  projection $T_{a_{\star}}(X) \to \mathsf T_\star$ whose kernel is generated by $v(a_\star)$. 

For a traversing $v$ and a majority  of trajectories, we can choose the space $T_{a_\star}(\d_1^+X)$ for the role of $\mathsf T_\star$, where $a_\star$ is the lowest point of $\g \cap \d_1X$. 
\smallskip

The configuration $\{\mathsf T_i\}$ of \emph{affine} subspaces  $\mathsf T_i \subset \mathsf T_\star$ is called \emph{generic} (or \emph{stable}) when all the multiple intersections of spaces from the configuration have the least possible dimensions, consistent with the dimensions of $\{\mathsf T_i\}$. In other words, $$\textup{codim}(\bigcap_{s} \mathsf T_{i_s},  \mathsf T_\star) = \sum_s \textup{codim}(\mathsf T_{i_s},  \mathsf T_\star)$$ for any subcollection $\{\mathsf T_{i_s}\}$ of spaces from the list $\{\mathsf T_i\}$.

Consider the case when $\{\mathsf T_i\}$ are \emph{vector} subspaces of $\mathsf T_\star$.  If we interpret each $\mathsf T_i$ as the kernel of a linear epimorphism $\Phi_i:  \mathsf T_\star \to \R^{n_i}$, then the property of  $\{\mathsf T_i\}$ being generic can be reformulated as the property of  the direct product map $\prod_i \Phi_i:  \mathsf T_\star \to \prod_i  \R^{n_i}$ being an epimorphism.  In particular, for a generic configuration of affine subspaces, if a point  belongs to several  $\mathsf T_i$'s, then the sum of their codimensions $n_i$ does not exceed the dimension of the ambient space $\mathsf T_\star$. 
\smallskip

The definition below resembles the ``Normal Crossing Condition" imposed on Boardman maps between smooth manifolds (see \cite{GG}, page 157, for the relevant definitions). In fact, for  traversing generic fields $v$, the $v$-flow delivers germs of Boardman maps $p(v, \g): \d_1X \to \R^n$, available in the vicinity of every trajectory $\g$. 

\begin{definition}\label{def1.3} We say that a traversing field $v$ on $X$ is \emph{traversally generic} if: 
\begin{itemize}
\item  the field is boundary generic in the sense of Definition \ref{def1.1},
\item for each $v$-trajectory $\g \subset X$ (not a singleton), the collection of  subspaces $\{\mathsf T_i(\g)\}_i$  is generic in $\mathsf T_\star$: that is, the obvious quotient map $\mathsf T_\star \to \prod_i \big(\mathsf T_\star/ \mathsf T_i(\g)\big)$ is surjective.   
\end{itemize}

We denote by $\mathcal V^\ddagger(X)$ the space of all traversally generic fields on $X$. \hfill\qed
\end{definition}

\noindent{\bf Remark 1.1.} In particular, the second bullet of the definition implies the inequality $$\sum_i \textup{codim}(\mathsf T_i(\g), \mathsf T_\star) \leq \dim(\mathsf T_\star) = n.$$   In other words, for traversally generic fields, the reduced multiplicity of each trajectory $\g$ satisfies the inequality
\begin{eqnarray}\label{eq3.4}
m'(\g) = \sum_i (j_i - 1) \leq n.
\end{eqnarray}
\hfill\qed

The following key lemma (see Lemma 3.4 from \cite{K2}) provides us a a semi-local analytic description of the traversally generic fields.

\begin{lemma}\label{lem1.2} Let  $v$ be a traversing generic field on $X$ and $\g$ its trajectory such that the intersection $\g \cap \d_1X$ is a  union of several points $a_i \in \d_{j_i}X(v)^\circ$. 

Then $\g$ has a $\hat v$-adjusted neighborhood $V$ with a special system of coordinates 
$$(u,\, \underbrace{x_{10}, \dots  , x_{1j_1-2}},\, \dots \, ,\underbrace{x_{i0}, \dots , x_{ij_i - 2}}, \, \dots \, ,\underbrace{x_{p0}, \dots \, x_{pj_p - 2}},\, \underbrace{y_1, \dots, y_{n - m'(\g)}})$$
such that:
\begin{itemize}
\item $\{u = const\}$ defines a transversal section of the $\hat v$-flow,
\item each $\hat v$-trajectory in $V$ is produced by fixing all the coordinates $\{x_{il}\}$ and $\{y_k\}$,
\item there is $\e > 0$ such that $V \cap \d_1X \subset \coprod_i V_i$, where  $V_i :=  u^{-1}((\a_i -\e, \a_i + \e)) \cap V$,  and $\a_i = u(a_i)$,
\item  the intersection  $V \cap \d_1X$ is given by the equation
\begin{eqnarray}\label{eq1.7}
\prod_i \big[(u - \a_i)^{j_i} + \sum_{l=0}^{j_i - 2} x_{i, l} (u - \a_i)^l\big]  = 0. 
\end{eqnarray}
 \hfill\qed
\end{itemize}
\end{lemma}

\section{Bifurcations of the Real Polynomial Divisors --- the Combinatorics of Tangency for Traversally Generic Flows}

In \cite{K2}, we have seen evidence that, for smooth traversally generic fields $v$ on a compact manifold $X$, the combinatorial patterns of tangency along $v$-trajectories resemble the combinatorics of divisors in $\R$, produced by real polynomials of an even degree $d \leq 2\, dim(X)$ (see Lemma 3.4 , Theorems 3.1, and Theorem 3.5 from  \cite{K2}). Now we are going to devote time to somewhat involved investigations of this combinatorics.
\smallskip

For any  polynomial $P(z)$  with real coefficients, we denote by $D_\R (P)$ its divisor in $\R$, and by $D_\C (P)$ its complex conjugation-invariant divisor in $\C$. Points in $\sup(D_\R (P))$, the support of $D_\R (P)$,  inherit the natural order from $\R$. The set $\sup(D_\C (P)) \subset \C$ is invariant under the complex conjugation. 

Let $\mathcal D_d$ denote the space of all divisors $D$ in $\R$ of degree $|D| = d$. The space $\mathcal D_d$ can be identified with 
the $d$-th symmetric power $\mathsf{Sym}^d(\R)$ of the real number line $\R$. Alternatively,  $\mathcal D_d$ can be introduced as the domain $\Pi_d$ in $\R^d$ given by the inequalities $x_1 \leq x_2 \leq \dots \leq x_d$ imposed on the coordinates $(x_1, x_2, \dots , x_d)$.

The divisors $D \in \mathcal D_d$ have  \emph{combinatorial models} represented by maps $$\omega_D: \{1, 2, 3, \dots \} \to \{0, 1, 2, 3, \dots \}$$, where $\omega_D(i) \geq 1$ is the multiplicity of the $i$-th point in $\sup(D)$, so that $\sum_i \omega_D(i) = d$. 

Let $\N$ be the set of all natural numbers and $\Z_+$ the set of all non-negative integers. Consider the set $\Omega$ of all maps $\omega: \N \to \Z_+$ with finite support and such that $\omega(i) \neq 0$ implies $\omega(j) \neq 0$ for all $j < i$. \smallskip

Consider a real polynomial $P$ of degree $d$ and its complex divisor $D_\C(P)$. There is $\e > 0$ such that the $\e$-neighborhood $U_\e$ of the support $\sup(D_\C(P)) \subset \C$  is a union of \emph{disjoint}  open disks. Then any real $d$-polynomial $Q$, sufficiently close to $P$, will have all its roots residing in $U_\e$. 

The radial contraction of each disk from $U_\e$ to its center commutes with the complex conjugation in $\C$ and defines a conjugation-equivariant deformation of the divisor $D_\C(Q)$ into the divisor $D_\C(P)$.  This produces a deformation of $Q$ to $P$, a curve $Q_s$, $s \in [0, 1]$, in the space of real polynomials of degree $d$. At all stages of  this deformation, but the last one ($s = 1$), the divisor $D_\C(Q_s)$  has the same multiplicity pattern as the one of $D_\C(Q)$; moreover, the divisor $D_\R(Q_s)$  has the same $\R$-ordered multiplicity pattern as the one of $D_\R(Q)$. 

Now imagine this deformation process from the viewpoint of an observer residing in $\R$, so that morphings of all the roots residing in $\C \setminus \R$ are ``invisible".  

Deformations of $P$ within the space of real polynomials of degree $\deg(P)$ change  its real divisor $D_\R(P)$ by sequences of two \emph{elementary operations} and their inverses: 
\begin{enumerate}
\item merging of two adjacent points from the support $\sup(D_\R (P))$ (their multiplicities add up); 
\item inserting a point of multiplicity 2 to the set $\R \setminus \sup(D_\R (P))$.
\end{enumerate}
The second elementary operation corresponds  to a pair of simple complex-conjugate roots merging at a point of $\R \subset \C$. 

These operations have combinatorial analogues. We define the elementary merge operation $\mathsf M_j: \Omega \to \Omega$ 
by the formula
\begin{eqnarray}\label{eq2.1} 
\mathsf M_j(\omega)(i) & = & \omega(i) \; \;\text{for \;all} \; i < j,\\ \nonumber
\mathsf M_j(\omega)(j)  & = & \omega(j) + \omega(j+1),\\ \nonumber
\mathsf M_j(\omega)(i)  & = & \omega(i + 1) \;\; \text{for \;all} \; i > j +1,
\end{eqnarray}
where $\omega \in \Omega$ and $1 \leq j \leq |\sup(\omega)|$. 

Define the elementary insert operation $\mathsf I_j: \Omega \to \Omega$ by the formula
\begin{eqnarray}\label{eq2.2}  
\mathsf I_j(\omega)(i) & = & \omega(i) \; \;\text{for \;all} \; i < j, \\ \nonumber
\mathsf I_j(\omega)(j) & = & 2,\\ \nonumber
\mathsf I_j(\omega)(i) & = & \omega(i - 1) \;\; \text{for \;all} \; i > j \geq 1, 
\end{eqnarray}
and 
\begin{eqnarray}\label{eq2.3}  
\mathsf I_0(\omega)(1) & = & 2 \\ \nonumber
\mathsf I_0(\omega)(i) & = & \omega(i -1), \text{for \;all} \; i > 1.
\end{eqnarray}

We also allow for the merge of several adjacent points in $\sup(D_\R (P))$ and for the merge of groups of conjugate roots (of any multiplicity) at a point(s) of $\R$. However, we tend to view these  morphings as compositions of the elementary operations $\{\mathsf M_j\}$ and $\{\mathsf I_j\}$.  

Note that the \emph{parity} of the degree $\deg(D_\R (P))$ is preserved under the merge and insert  operations. 
\smallskip

Guided by Definition \ref{def1.2}, we introduce combinatorial analogues of the quantities from that definition:
\begin{itemize}
\item the $l_1$-\emph{norm} $| \omega | $ of $\omega$ by the formula $\sum_i \omega(i)$,\footnote{This number  represents $\deg(D_\R(P))$.} 
\item the \emph{reduced norm} $|\omega |'$ of $\omega$, by the formula  $\sum_i (\omega(i) -1)$,
\item the \emph{virtual multiplicity} $\mu(\omega)$ of $\omega$,  by the formula $\sum_i \lceil \omega(i)/2 \rceil$, where $\lceil \sim \rceil$ denotes the integral part of a real number.
\end{itemize}
\smallskip

Now  we are in position to define, via the merge and insert operations, a \emph{partial order} "$\succ$" in the set $\Omega$: 

\begin{definition}\label{def2.1} 
For $\omega_1, \omega_2 \in \Omega$, we write $\omega_1 \succ \omega_2$ if $\omega_2$ cam be obtained from $\omega_1$ by a sequence of  merge  operations $\{\mathsf M_j\}$ and insert operations $\{\mathsf I_j\}$ as in (\ref{eq2.1})-(\ref{eq2.3}). \hfill\qed
\end{definition}

Let  $\Omega_{d} \subset \Omega$ denote the set of all $\omega$'s such that $|\omega| = d$. Let  $\Omega_{\langle d\,]} \subset \Omega$ denote the set of all $\omega$'s such that $|\omega| \leq d$ and $|\omega| \equiv d \; (2)$. The set $\Omega_{\langle d\,]}$ inherits its partial order from the ambient poset $(\Omega, \succ)$. 

For all our applications we will  need only the case of an even $d$.
\smallskip

We can give an interpretation to the poset $(\Omega, \succ)$ in the spirit of category theory. In this interpretation, the elements of $\Omega$ become objects of some category $\mathbf \Omega$ so that the relation $\omega_1 \succ \omega_2$ is transformed into the property $Mor(\omega_1, \omega_2) \neq \emptyset$. 
\smallskip

The following definition mimics the map of conjugation-invariant divisors $D_\C(Q) \to D_\C(P)$, induced by the radial retraction of the $\e$-neighborhood $U_\e(\sup(D_\C(P)))$ to its core, $\sup(D_\C(P))$, as observed from within $\R \subset \C$. 

\begin{definition}\label{def2.2}
For any two elements $\omega_1, \omega_2 \in \Omega$ we define the set $Mor(\omega_1, \omega_2)$ as the set of maps $\a: \sup(\omega_1) \to \sup(\omega_2)$ such that: 
\begin{enumerate}
\item for each pair $i < i'$ in $\sup(\omega_1)$,  $\a(i) \leq \a(i')$, 
\item $\sum_{j \in \a^{-1}(i)}\, \omega_1(j) \leq \omega_2(i)$ for all $i \in \sup(\omega_2)$\footnote{This restriction is vacuous when $\a^{-1}(i) = \emptyset$.},
\item $\sum_{j \in \a^{-1}(i)}\, \omega_1(j) \equiv \omega_2(i)\; \mod (2)$ for all $i \in \sup(\omega_2)$\footnote{In particular, if $\a^{-1}(i) = \emptyset$, then $\omega_2(i)  \equiv 0\, \mod(2).$}. \hfill\qed
\end{enumerate}
 
\end{definition}

For example, $Mor((121), (11) = \emptyset$, while  $Mor((11), (121))$ consists a single injective map.   
\smallskip

Examining properties $(1)$-$(3)$ above, we see that there is a natural pairing $$Mor(\omega_1, \omega_2) \times Mor(\omega_2, \omega_3) \to Mor(\omega_1, \omega_3)$$ defined by the composition of maps, provided $Mor(\omega_1, \omega_2) \neq \emptyset$ and $Mor(\omega_2, \omega_3) \neq \emptyset$.

\begin{lemma}\label{lem2.1} In the category $\mathbf \Omega$, the set $Mor(\omega_1, \omega_2) \neq \emptyset$, if and only if, $\omega_1 \succeq \omega_2$ in the poset $(\Omega, \succ)$.
\end{lemma}

\begin{proof} Any elementary operation $\mathsf M_j$ in (\ref{eq2.1}) gives rise to an element  $\mu_j \in Mor(\omega, \mathsf M_j(\omega))$ which maps the pair $j, j+1 \in \sup(\omega)$  to the single element $j \in \sup(\mathsf M_j(\omega))$; the rest of the elements in $\sup(\omega)$ are mapped bijectively by $\mu_j$. Evidently, the properties $(1)$-$(3)$ in Definition \ref{def2.2} are satisfied. Similarly, any elementary operation $\mathsf I_j$ in  (\ref{eq2.2}) gives rise to an element  $\nu_j \in Mor(\omega, \mathsf I_j(\omega))$ which maps $\sup(\omega)$ bijectively and in a monotone fashion to $\sup(\mathsf I_j(\omega))$ so that $j \in \sup(\mathsf I_j(\omega))$ is the only element that is not in the image of $\nu_j$. Again, the properties $(1)$-$(3)$ in Definition \ref{def2.2} are satisfied. 

Therefore if two elements, $\omega_1$ and $\omega_2$, are linked by a sequence of elementary operations of the types $\mathsf M_j$ and $\mathsf I_k$, then there exists an element  $\a \in Mor(\omega_1, \omega_2)$ which is obtained by composing the chain of corresponding maps $\mu_j$ and $\nu_k$. Hence,  $Mor(\omega_1, \omega_2) \neq \emptyset$.
\smallskip

On the other hand, if $Mor(\omega_1, \omega_2) \neq \emptyset$, then any  $\a \in Mor(\omega_1, \omega_2)$ can be obtained in such a way via elementary operations. Indeed, for each $i \in \sup(\omega_2)$, consider the set $\a^{-1}(i)$. Put $$a_i = \omega_2(i) - \sum_{j \in \a^{-1}(i)}\, \omega_1(j).$$ By Definition \ref{def2.2}, $a_i \geq 0$ and $\a_i \equiv 0\; \mod(2)$. We apply a sequence of $a_i/2$ insert operations $\mathsf I_k$ to $\omega_1$, localized to the set $\a^{-1}(i)$. They will add $a_i/2$ copies of $2$'s to the sequence $\omega_1$. The location of these $2$'s relative to the elements of $\a^{-1}(i)$ is unimportant. The resulting $\omega'_1$ now has the property $Mor(\omega_1, \omega'_1) \neq \emptyset$. These insertions $\{\mathsf I_k\}$ define a map $\g \in Mor(\omega_1, \omega'_1)$.  Next we merge all the original elements of the set $\a^{-1}(i)$ and the locations of newly inserted $2$'s together into a singleton $i_\star$ by a sequence of elementary merges $\{\mathsf M_j\}$. Again,  the order in which the elementary merges are performed is unimportant. The resulting $\omega''_1$ is such that there exists $\delta \in  Mor(\omega_1, \omega''_1)$ with the property $$\sum_{j \in \delta^{-1}(i_\star)}\, \omega_1(j) = \omega''_1(i_\star) = \omega_2(i).$$

Finally, we apply this procedure to every element $i \in \sup(\omega_2)$. The result of these operations transforms $\omega_1$ into $\omega_2$ by a sequence of elementary operations.  Therefore, $\omega_1 \succ \omega_2$ in the poset $\Omega$.
\end{proof}
 \smallskip

The polynomial inequality $P(z) \leq 0$ splits the support $\sup(D_\R(P))$ into a number of disjoint sets: each set is formed by the maximal string of consecutive roots so that, in the closed interval bounded by the maximal and the minimal root from the string, the inequality  $P(z) \leq 0$ is valid.  For a polynomial of an even degree, each maximal string of roots either is a singleton whose multiplicity is even, or a sequence whose maximal and minimal elements have odd multiplicities, while the rest of roots have even multiplicities. These observations motivate the following combinatorial models.
\begin{definition}\label{def2.3}
Let  $\Omega^\bullet$ denote the set of maps $\omega \in \Omega$ that satisfy the following properties:
\begin{itemize}
\item either
\begin{enumerate}
\item $\omega(1)$, $\omega(q)$ are odd numbers, where $q = |\sup(\omega)|$, and
\item $\omega(i)$ is even for $1 < i < q$;
\end{enumerate}
\item or  $\sup(\omega) = \{1\}$, and $\omega(1) \equiv 0\; \mod(2)$.  \hfill\qed
\end{itemize} 
\end{definition}
A priori $q$,  the cardinality of the support of $\omega \in \Omega^\bullet$, is not fixed.
\begin{definition}\label{def2.4}
Let  $\Omega^\bullet_{'d}$ denote the set of maps $\omega \in \Omega^\bullet$ such that $|\omega|' := \sum_i (\omega(i) - 1) = d$. 

Let $$\Omega^\bullet_{'[k, d]}  := \coprod_{k \leq j \leq d} \Omega^\bullet_{'j}.$$ We also will use the shorter notation ``$\Omega^\bullet_{'\langle d]}$" for the set $\Omega^\bullet_{'[0, d]}$.  \hfill\qed
\end{definition}

Recall that on a $(n+1)$-manifold $X$, any  trajectory $\g$ of a boundary generic field $v \in \mathcal V^\dagger(X)$ gives rise to an element $\omega_\g \in \Omega^\bullet$  according to the rule:  $\omega_\g(i) = m(a_i)$, the multiplicity of the $i$-th point $a_i$ in the $v$-ordered set $\g \cap \d_1X$\footnote{that is, $a_i \in \d_{m(a_i)}X(v)^\circ$.}. By Theorem 3.5 from \cite{K2}, for any traversally generic field $v \in \mathcal V^\ddagger(X)$, we get $|\omega_\g|'  \leq n = \dim(\d_1X)$, so that  $\omega_\g \in \Omega^\bullet_{'\langle n]}$.
\smallskip

In the end of the proof of Theorem 3.5 from \cite{K2}, we have established the following proposition: 

\begin{lemma}\label{lem2.2}
For $\omega \in  \Omega^\bullet_{'\langle n\,]}$, we have  $|\omega| := \sum_i \omega_\g(i) \leq 2n + 2$. \hfill\qed
\end{lemma}



We are going to introduce a partial order among the elements of the set $\Omega^\bullet_{' \langle d\,]}$, which will match the changing geometry of orbits for traversally generic fields. Crudely, the order is induced from the ambient poset $\Omega \supset \Omega^\bullet_{' \langle d\,]}$, but then enhanced. For a more accurate description, we turn to few auxiliary combinatorial constructions.     

Given any $\omega \in \Omega$, we would like to ``chop it"  into a number of ``strings" and ``atoms":  each string belongs to some  
$\Omega^\bullet_{'d}$ as in the first bullet of the Definition \ref{def2.3}, while each atom has a singleton for its support and takes there an even value (see Fig. 1). Prior to defining this canonical deconstruction $\Xi(\omega)$ of $\omega$,  we need to introduce some notations. 

Let $\omega \in \Omega$ be such that $|\omega| \equiv 0 \; (2)$.  Denote by $\sup_{odd}(\omega)$ the  points $l \in \N$ in the support of $\omega$ such that $\omega(l) \equiv 1\; (2)$ and by $\sup_{ev}(\omega)$ the  points $l$ in the support of $\omega$ such that $\omega(l) \equiv 0\; (2)$.  In the geometrical context of traversing flows, the number  $|\sup_{odd}(\omega)|$ is even. We count the elements of $\sup_{odd}(\omega)$ as they appear in the list; some of them acquire odd numerals, others acquire even ones. A \emph{string} in $A \subset \sup(\omega)$ is formed by all the elements that are bounded on the left by an element of  $i \in \sup_{odd}(\omega)$ with an odd numeral and on the right by the next element $j \in \sup_{odd}(\omega)$ (it has an even numeral attached to it). In other words, a string $A := [i, j]$ is formed by such $i, j$ and all the elements from $\sup_{ev}(\omega)$ that lie in-between. The  points from $\sup_{ev}(\omega)$ that  do not belong to any string are called ``\emph{atoms}". We can compute the values of $\omega$ at the elements of each string $A$ as well as at each atomic support. This gives rise to a unique ordered sequence $\Xi(\omega)$ of strings,  interrupted by a number of atoms.

\begin{figure}[ht]
\centerline{\includegraphics[height=0.8in,width=2.8in]{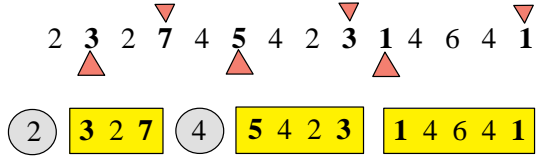}}
\bigskip
\caption{\small{An example of the deconstruction $\omega \Rightarrow \Xi(\omega)$. The strings are boxed, the atoms are circled.}}
\end{figure}

\begin{definition}\label{def2.5} We define the partial order ``$\prec_\bullet$" in $\Omega^\bullet$ as follows:  for $\omega_1, \omega_2 \in \Omega^\bullet$, we write ``$\omega_1\, \prec_\bullet \, \omega_2$" if $\omega_2$ occurs as a string or an atom from the list $\Xi(\omega)$ for some $\omega \succ \omega_1$, where $\omega \in \Omega$ and the last ordering is considered in the poset $(\Omega, \succ)$.  \hfill\qed
\end{definition}

\noindent{\bf Example 2.1.} Consider the poset $(\Omega^\bullet_{' [0, 3]}, \prec_\bullet )$ shown in Fig. 2. Then $(1,4,1) \prec_\bullet (3,1)$ since $(3, 1)$ is present as a  string in $(1, 1, 3, 1) \succ (1, 4, 1)$, the order "$\succ$"  being the one from the poset $\Omega_{[0, 8]}$. \hfill\qed
\smallskip

\begin{figure}[ht]
\centerline{\includegraphics[height=1.8in,width=3in]{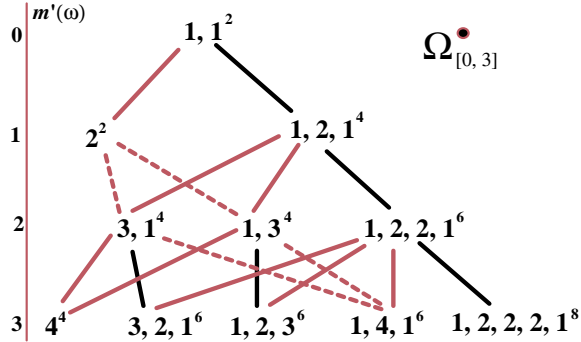}}
\bigskip
\caption{\small{The poset $(\Omega^\bullet_{' [0, 3]}, \prec_\bullet )$ with the "height" function $m'(\omega) := |\omega|'$. The bold dark lines indicate the insert operations, the lighter bold lines the merge operations. The dotted lines represent the order relations, as in Definition \ref{def2.5}, that are not directly induced from the poset $\Omega_{[0, 8]} \supset \Omega^\bullet_{' [0, 3]}$.  The upper index next to $\omega$ shows the value of the norm $|\omega|$.}}
\end{figure}

\begin{lemma}\label{lem2.3} For any $d$, the sub-poset of $\Omega^\bullet_{'\langle d]}$, defined by the constraint $|\omega|' < d$, is canonically isomorphic to the poset   $\Omega^\bullet_{'\langle d-1]}$.  As a result, the direct limit $$\lim_{d \to +\infty} \Omega^\bullet_{'\langle d]} = \Omega^\bullet$$ as posets.
\end{lemma}

\begin{proof} Let  $\omega_1\, \succ_\bullet \omega_2$ in $\Omega^\bullet_{'\langle d]}$. This implies that $\omega_1$ is string or an atom in some  $\omega \succ \omega_2$ in $(\Omega, \succ)$. If  $|\omega_2|' < d$,  then $|\omega|' < d$ as well. In turn, $|\omega|' < d$ implies that $|\omega_1|' < d$. Therefore $\omega_1\,  \succ_\bullet \omega_2$ in $\Omega^\bullet_{'\langle d - 1]}$ as well.
\end{proof}

Similar to the poset $(\Omega, \succ)$,  the poset $(\Omega^\bullet, \succ_\bullet)$ allows for an interpretation in terms of a category theory $\mathbf \Omega^\bullet$ whose objects are elements of $\Omega^\bullet$. 

\begin{definition}\label{def2.6} For any two elements $\omega_1, \omega_2 \in \Omega^\bullet$ we define $Mor^\bullet(\omega_1, \omega_2)$ as the set of maps $\a: \sup(\omega_1) \to \sup(\omega_2)$ such that: 
\begin{enumerate}
\item for each pair $i < i'$ in $\sup(\omega_1)$,  $\a(i) \leq \a(i')$, 
\item $\sum_{j \in \a^{-1}(i)}\, \omega_1(j) \leq \omega_2(i)$ for all $i \in \sup(\omega_2)$\footnote{This restriction is vacuous when $\a^{-1}(i) = \emptyset$.}.  \hfill\qed
\end{enumerate}
\end{definition}

\begin{lemma}\label{lem2.4} In the category $\mathbf \Omega^\bullet$, the set $Mor_\bullet(\omega_1, \omega_2) \neq \emptyset$, if and only if, $\omega_1 \succeq_\bullet \omega_2$ in the poset  $(\Omega_\bullet, \succ_\bullet)$.
\end{lemma}

\begin{proof}  If $\omega_1 \succeq_\bullet \omega_2$, then by the definition of the order $\succ_\bullet$,  there exists an element $\tilde\omega_1\in \Omega$ such that $\tilde\omega_1 \succ \omega_2$, and $\omega_1$ is a string or an atom in $\tilde\omega_1$. By Lemma \ref{lem2.1}, $Mor(\tilde\omega_1, \omega_2) \neq \emptyset$. When we restrict a map $\tilde\a \in Mor(\tilde\omega_1, \omega_2)$ to the subset $\sup(\omega_1) \subset \sup(\tilde\omega_1)$ we produce a map $\a: \sup(\omega_1) \to \sup(\omega_2)$. Under this restriction, the properties $(1)$ and $(2)$ from Definition \ref{def2.2}, valid for $\tilde\a$, are evidently preserved for its restriction $\a$: the support of a string consists of consecutive indices, and $\tilde\a^{-1}(i) \supset \a^{-1}(i)$ for all $i \in \sup(\omega_2)$, while the property $(3)$ could be violated. Thus $\a \in Mor^\bullet(\omega_1, \omega_2)$, a nonempty set.
\smallskip

On the other hand, if $Mor^\bullet(\omega_1, \omega_2) \neq \emptyset$, then it is possible to construct $\tilde\omega_1 \in \Omega$ such that $Mor(\tilde\omega_1, \omega_2) \neq \emptyset$. Indeed, consider some $\a: \sup(\omega_1) \to \sup(\omega_2)$ having properties $(1)$ and $(2)$ from Definition \ref{def2.6}. If, in addition,  property $(3)$ from Definition \ref{def2.2} holds, we are done. Note that this parity property can only be violated if one or both ends $\{1\}, \{q\}$ in the support of  the string $\omega_1$ are mapped by $\a$ to the location $\a(1)$ or $\a(q)$ such that $\omega_2(\a(1)) \equiv 0\; \mod(2)$, or $\omega_2(\a(q)) \equiv 0\; \mod(2)$. 

Let us consider the case $\omega_2(\a(1)) \equiv 0\; \mod(2)$. The monotonisity of $\a$ implies that the minimum element $\{1_\star\}$ in the support of $\omega_2$ is not in the image of $\a$. We can append  a string $\omega' = (1,1)$ below (to the left) of the string $\omega_1$ to form a new element $\tilde\omega_1 \in \Omega$ (not a string!) so that 
$\sup(\tilde\omega_1) = \sup((1,1)) \sqcup \sup(\omega)$. Then we extend the map $\a$ to a map $\tilde\a: \sup((1,1)) \sqcup \sup(\omega) \to \sup(\omega_2)$ by sending the first element of the extended support to the minimal element $\{1_\ast\} \in \sup(\omega_2)$, and the second and the third elements both to $\a(1)$ (note that the third element comes from the element $\{1\} \in \sup(\omega_1)$). This will repair the parity defect of the original $\a$. A similar treatment applies when $\omega_2(\a(q)) \equiv 0\; \mod(2)$. Thus, $Mor^\bullet(\omega_1, \omega_2) \neq \emptyset$ implies that $Mor(\tilde\omega_1, \omega_2) \neq \emptyset$. 

By Lemma \ref{lem2.1}, we get that $\tilde\omega_1 \succ \omega_2$ in $(\Omega, \succ)$.  Therefore,  $Mor^\bullet(\omega_1, \omega_2) \neq \emptyset$ implies  that $\omega_1 \succ_\bullet \omega_2$.
\end{proof}
\smallskip

The following lemma can be viewed as an a posteriori justification for introducing the poset $(\Omega^\bullet, \succ_\bullet)$. Here,  for a boundary generic vector field $v \in \mathcal V^\dagger(X)$ and its trajectory $\g$, we view the set $\g \cap \d_1X$, together with the multiplicities attached to its points, as a divisor $D_\g$ on $\g$. 

\begin{lemma}\label{lem2.5} Let $v \in \mathcal V^\dagger(X)$. If a sequence $\{ x_k \in X\}_k$ converges to a point $x$, then $$|D_{\g(x)}|' \; \geq \; \overline{\lim}_{k \to +\infty}\; |D_{\g(x_k)}|'$$, i.e. the reduced norm $|D_{\g( \sim)}|' $ is a upper semi-continuous function on $X$. Moreover, if the divisors $\{D_{\g(x_k)}\}_k$ all share a combinatorial type $\omega$, and $\omega_1 \neq \omega$ is the combinatorial type of  $D_{\g(x)}$, then we get: $\omega  \succ_\bullet \omega_1$,\,   $|\omega_1| \geq |\omega|$,\, and $|\omega_1|' > |\omega|'$.
\end{lemma}
\begin{proof} The trust of the argument is that any merge or insert operations (see (\ref{eq2.1})-(\ref{eq2.3})) that change the combinatorics of $D_\g$  cannot decrease the multiplicity $m(\g)$, as well as the reduced multiplicity $m'(\g)$ (defined in (\ref{eq1.2})-(\ref{eq1.4})). In fact, any such operation does increase $m'(g)$. 

In the argument below, we use our understanding of local models (\ref{eq1.5}) (see Lemma \ref{lem1.1}) of fields $v \in \mathcal V^\dagger(X)$ in the vicinity of $\d_1X$. They imply that, as we vary $\g$, its divisor $D_\g$ (with the support  $\g \cap \d_1X$) can only change locally  as the real divisors of  the family of $u$-polynomials  in (\ref{eq1.5}) do: that is, via resolutions that either preserve their degree, or drop it by an even number. By the same token, as $\{x_k\}$ converge to $x$, in the vicinity of $x$, the divisor $D_{\g(x)}$ is obtained from $\{D_{\g(x_k)}\}$ only via:  (1) merging adjacent distinct points from $\sup(D_{\g(x_k)})$, or (2) by inserting new points of even multiplicity, or 3) by increasing the odd multiplicity of the end point from $\sup(D_{\g(x_k)})$ by an odd number (this happens when the ends of two distinct trajectories merge at $x$).

Let us take a closer look at these mechanisms.  

By Lemma \ref{lem1.1}, for a boundary generic $v \in \mathcal V^\dagger(X)$, the multiplicity of each point from $\g \cap \d_1X$ does not exceed $n +1$ for all $\g$'s. Moreover, by the same lemma and a compactness argument, no trajectory $\g$ has an infinite set $\g \cap \d_1X$. Thus the multiplicity $m(\g) < \infty$ for each $\g$. 

Consider  any infinite set of trajectories $\{\g_k\}_k$ such that  the set of multiplicities $\{m(\g_k)\}_k$ is unbounded. Since the multiplicity of each point of tangency is uniformly bounded, this implies that the set $$A := \cup_{1 \leq k < \infty} (\g_k \cap \d_1X)$$ is infinite and $\lim_{k \to \infty} |\g_k \cap \d_1X| = \infty$. Let $y$ be a limit point for $A$. Consider the trajectory $\g_y$ through $y$. Using  the compactness of $\g_y$ and the local models from Lemma \ref{lem1.1}, $\g_y$ has a neighborhood in which any trajectory $\tilde\g$ has a uniformly bounded cardinality of the set $\tilde\g \cap \d_1X$.  This  contradicts to the assumption that $\lim_{k \to \infty} |\g_k \cap \d_1X| = \infty$. Therefore there is a number $d$ so that, for every trajectory $\g$ in $X$, its combinatorial type belongs to the set $\Omega_{\langle d]}$. 

Consider the combinatorial types of $\{D_{\g(x_k)}\}_k$, a finite set. At least one of them, say $\omega$,  must occur infinitely many times in a subsequence  $\{D_{\g(x_{k_i})}\}_i$ of $\{D_{\g(x_{k})}\}_k$. For each $j \in \sup(\omega)$, consider the sequence  of points $\{y_{j,i} \in \g(x_{k_i})\}_i$ which occupy the $j$-th position in $\g(x_{k_i}) \cap \d_1X$. By its definition, $y_{j,i} \in \d_{\omega(j)}X^\circ$. Since $\d_{\omega(j)}X$ is compact,  there exists a subsequence $\{y_{j,i_l} \}_l$ that converges to a point $y_j \in \d_{\omega(j)}X$. Its multiplicity, $m(y_j)$, is $\omega(j)$ at least. 

Consider the segments $\g_{j, j+1; i}$ of $\g(x_{k_i})$ bounded by the adjacent pairs $(y_{j,i}, y_{j+1, i})$. The interiors of the segments do not intersect the boundary $\d_1X$. By the continuous dependence of $v$-integral curves on their end points,  $\{\g_{j, j+1; i_l}\}_l$ converge to a trajectory segment $\g_{j, j+1}$ that connects $y_j$ with $y_{j+1}$ (that segment may contain  "new" points from $\d_1X$ in its interior). As a result, all the $y_j$'s belong to the same trajectory $\g$. Since each $x_{k_{i_l}}$ belongs to some segment $\g_{j, j+1; i_l}$ and $\lim_{l \to +\infty} x_{k_{i_l}} = x$, we get $\g(x) = \g$. 

By the argument above, $\sup(D_{\g(x)})$ may contain new points that are not in the limit of  $\cup_l\, \sup(D_{\g(x_{k_{i_l}})})$. This can only boost the reduced multiplicity $|D_{\g(x)}|'$ in comparison to $|D_{\g(x_{k_{i_l}})}|'$.  On the other hand, some points in $\sup(D_{\g(x)})$ are the result of merging of two or several consecutive points $\{y_{j, {i_l}}\}_j$ from $\g(x_{k_{i_l}})$. In the process, thanks to the local models  (\ref{eq1.5}), their multiplicities add; so again, the reduced multiplicity of the limiting point exceeds the sum of the reduced  multiplicities of the corresponding merging points. Therefore, $$m'(\g(x)) \geq \overline{\lim}_{k \to +\infty}\; m'(\g(x_k)).$$ Moreover, if the combinatorial type $\omega$ of $\g(x_{k_i}) \cap \d_1X$ differs from the combinatorial type $\omega'$ of $\g(x) \cap \d_1X$, then  $m'(\g(x)) >  \overline{\lim}_{k \to +\infty}\; m'(\g(x_k))$. Similarly,  $m(\g(x)) \geq \overline{\lim}_{k \to +\infty}\; m(\g(x_k))$. 

The same same argument shows that the relation between the types $\omega$ and $\omega'$ is exactly the order relation $\omega \succ_\bullet \omega'$ in $\Omega^\bullet$, as introduced in Definition \ref{def2.5}.
\end{proof}

\section{On Stratified Spaces}

Let $(\mathcal S, \succ)$ be a \emph{poset} (a partially ordered set).  Given a subset $\Theta \subset \mathcal S$ in the poset $(\mathcal S, \succ)$, we denote by $\Theta_\succeq$ the set $$\big\{a \in  \mathcal S|\; b \succeq a \; \text{for some}\, b \in \Theta\big\}.$$
Let $\Theta_\succ := \Theta_\succeq \setminus \Theta$.  Similarly, we introduce $$\Theta_\preceq := \big\{a \in  \mathcal S|\; a \succeq b \; \text{for some}\, b \in \Theta\big\}\; \text{and}\; \Theta_\prec  := \Theta_\preceq \setminus \Theta.$$ 

In particular, for any $\omega \in \Omega^\bullet$, we will  consider routinely  the finite posets $\omega_{\succeq_\bullet}, \; \omega_{\succ_\bullet}$, and, for any $\omega \in \Omega$, the finite posets $\omega_{\succeq}, \; \omega_{\succ}$. 
\smallskip

\begin{definition}\label{def3.1}
A $\mathcal S$-\emph{filtration} of a topological space $X$ is a collection of closed topological subspaces $\{X_{\preceq\omega}\}_{\omega \in \mathcal S}$ such that $\omega \prec \omega'$ implies the inclusion $X_{\preceq\omega} \subset X_{\preceq\omega'}$. Moreover, we require that   each point $x \in X$ which belongs to some stratum $X_{\preceq\omega}$, $\omega \in \mathcal S$, belongs to a \emph{unique} smallest stratum $X_{\preceq\omega(x)}$ labelled by a \emph{minimal element} $\omega(x)$ from the poset $\{\preceq\omega\} \subset \mathcal S$.  \hfill\qed
\end{definition}

When $X$ itself is a member of the collection $\{X_{\preceq\omega}\}_{\omega \in \mathcal S}$, we assume that $\mathcal S$ has a unique maximal element $\omega_\star$ and that$X_{\preceq\omega_\star} = X$.

For an $\mathcal S$-filtrered space $X$, we define 
$$X_{\prec\omega} := \cup_{\omega' \prec \omega}\; X_{\preceq \omega'} \;\; \text{and} \;\; X_\omega := 
X_{\preceq \omega} \setminus X_{\prec \omega}.$$

In general, for any subset $\Theta \subset \mathcal S$, put $X_{\preceq\Theta} : = \cup_{\omega \in \Theta}\; X_{\preceq\omega}.$

\begin{definition}\label{def3.2} An $\mathcal S$-\emph{filtered map} $f: X \to Y$ between $\mathcal S$-filtered spaces $X$ and $Y$ is a continuous map such that $f(X_{\preceq\omega}) \subset Y_{\preceq\omega}$ for each $\omega \in \mathcal S$.\smallskip

A $\mathcal S$-\emph{filtered homotopy} between  $\mathcal S$-filtered maps $f_0: X \to Y$ and $f_1: X \to Y$ is an $\mathcal S$-filtered map $F: X\times [0, 1] \to Y$ such that $F(x, 0) = f_0(x)$ and $F(x, 1) = f_1(x)$ for all $x \in X$. \hfill\qed
\end{definition}

\section{Spaces of Real Polynomials, Stratified  by the Combinatorial Types of their Real Divisors}

Next, we would like to investigate carefully one natural stratification in  the coefficient space $\R^d_{\mathsf{coef}}$ of real monic polynomials $P(z)$ of a given degree $d$. This stratification is generated by different combinatorial patterns of  zero divisors $D_\R( P)$ in $\R$. Its importance for our program is justified by the local models for traversally generic fields that have been described in Lemmas \ref{lem1.1} and \ref{lem1.2}.  \smallskip

Let $\tau: \C \to \C$ be the complex conjugation. Via the Vi\`{e}te Map $V$, the coefficient space $\R^d_{\mathsf{coef}}$ can be identified with the space  $(\mathsf{Sym}^d\C)^\tau$ of $\tau$-invariant divisors in $\C$ of degree $d$. 

Each function $\omega: \N \to \Z_+$ of $l_1$-norm $|\omega| \leq d$ and such that  $|\omega| \equiv d \; (2)$ defines a pure stratum $(\mathsf{Sym}^d\C)_\omega^\tau$ in 
$(\mathsf{Sym}^d\C)^\tau$ and therefore, in $\R^d_{\mathsf{coef}}$. We denote this stratum $V(\mathsf{Sym}^d\C)_\omega^\tau)$ by  $\mathsf R_\omega$. In $\R^d_{\mathsf{coef}}$, it represents all monic real polynomial $P$ such that the combinatorics of $D_\R( P)$ is prescribed by $\omega$. We denote by $(\mathsf{Sym}^d\C)^{\tau}_{\omega_\succeq}$ the closure of the stratum  $[(\mathsf{Sym}^d\C)_\omega]^\tau$, and by $\mathsf R_{\omega_\succeq}$ the closure of the stratum $\mathsf R_\omega$, respectively.

The depressed form of the polynomial in (\ref{eq1.5}) and  (\ref{eq1.7}) calls for the introduction of similar spaces built out of, so called, \emph{balanced divisors}. 

Let $\a \in \R$. A divisor $D = \oplus_i m(i) z_i$ with $z_i \in \C$ is called $\a$-\emph{balanced}, if $$\sum_i m(i) z_i = \a\cdot\sum_i m(i)$$ in $\C$. In other words, $\a$ is the center of gravity of the configuration of points-particles in $\C$ representing $D$. In the root space $\mathsf{Sym}^d\C = \{z_1, \dots , z_d\}$ (where $d = |\omega|$), such divisors are characterized by the equation $\sum_{k =1}^d z_k = \a\cdot d$. 

The conjugation-invariant $\a$-balanced divisors from $(\mathsf{Sym}^d(\C))^\tau$ are described by the equation $$\sum_{k =1}^d \mathsf{Re}(z_k) = d\cdot \a.$$ They form a real hypersurface $(\mathsf{Sym}^d_\a\C)^\tau$ in $(\mathsf{Sym}^d\C)^\tau$. Zero-balanced divisors (i.e., $\a = 0$) are simply called \emph{balanced}. 

The image of $(\mathsf{Sym}^d_\a\C)^\tau$ under the Vi\`{e}te map $V$ consists of real monic polynomials whose $z^{d-1}$-coefficient is $\a$. In a similar way, we introduce the stratification $$\big\{\mathsf R_{\a, \omega} := V((\mathsf{Sym}^d_\a\C)_\omega^\tau)\big\}_\omega$$ in the  $(d -1)$-dimensional affine space $\R_{\mathsf{coef}, \a}^d$ of such polynomials. 

Let us stress that here we use $\omega$'s that attach multiplicities only to \emph{real} roots (their norms $|\omega|$ do not exceed $d$ and $|\omega| \equiv d\; (2)$)!
\smallskip

\begin{figure}[ht]
\centerline{\includegraphics[height=2.6in,width=2.6in]{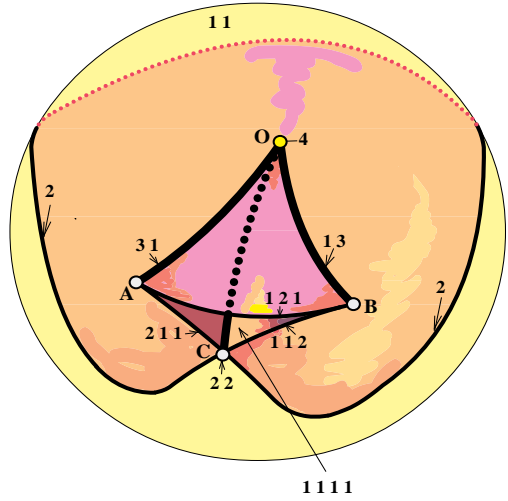}}
\bigskip
\caption{\small{The Swallow Tail singularity is linked to the Whitney projection of the hypersuface  $\{z^4 + x_2 z^2 + x_1 z + x_0 = 0\}$ onto the space $\R_{\mathsf{coef}, 0}^3$ with the coordinates $(x_0, x_1, x_2)$. The strata in  $\R_{\mathsf{coef}, 0}^3$ are indexed by elements $\omega \in \Omega_{\langle4]}$. They divide the target space into three 3-cells, four 2-cells, three 1-cells, and one 0-cell.}}
\end{figure}

Consider the upper half plane $$\H := \{z \in \C| \; \mathsf{Re}(z) \geq 0\}$$, and let $$\H^\circ := \{z \in \C| \; \mathsf{Re}(z) > 0\}.$$ Let $B^2 \subset \C$ be the unit ball centered on the origin. Put $B^2_+ = B^2 \cap \H$. Similarly, for any $\a \in \R$, let  $B^2(\a) \subset \C$ be the unit ball centered on $\a$, and $B^2_+(\a) :=  B^2(\a) \cap \H$.

\begin{lemma}\label{lem4.1} The symmetric product $\mathsf{Sym}^m\H$ is homeomorphic to the half-space $\R^{2m}_+$, bounded by a hyperplane in $\R^{2m}$. \smallskip

Let $B^2$ denotes a closed $2$-ball. The symmetric product $\mathsf{Sym}^mB^2$ is homeomorphic to a closed $2m$-ball $B^{2m}$. 
\end{lemma}

\begin{proof} First we validate the second claim of the lemma. 

The origin-centered dilation $t: \C \to \C$ induces an action $\psi_t$ on the ``root space" $\mathsf{Sym}^m\C$: under this action, the support of each divisor is scaled by the factor $t > 0$. With the help of the Vi\`{e}te homeomorphism $V: \mathsf{Sym}^m\C \to \C^m_{\mathsf{coef}}$, this $\psi_t$-action induces a $\Psi_t$-action on the space $ \C^m_{\mathsf{coef}}$.

The boundary of $\mathsf{Sym}^mB^2$ consists of divisors $D$  whose support is contained in $B^2$ and has a non-empty intersection with the circle $\d B^2$.  
We notice that each $\psi_t$-trajectory through a  point $D \in \mathsf{Sym}^mB^2$, distinct from the divisor $m\{0\}$, has a unique point of  intersection with the boundary $\d(\mathsf{Sym}^mB^2)$. Indeed, for each configuration $\D \neq \{0\}$ of points in $B^2$, there is a unique origin-centered dilation $t: \C \to \C$ such that $t(\D) \cap \d B^2 \neq \emptyset$ and $t(\D) \subset B^2$. Since the Vi\`{e}te map $V: \mathsf{Sym}^m\C \to \C^m_{\mathsf{coef}}$ is a smooth homeomorphism, $V(\mathsf{Sym}^mB^2) \subset \C_{\mathsf{coef}}^m$ is homeomorphic to $\mathsf{Sym}^mB^2$. Therefore each trajectory $\Psi_t$-trajectory, except for the trivial trajectory through $m\{0\}$, hits $\d V(\mathsf{Sym}^mB^2) = V(\d(\mathsf{Sym}^mB^2))$ transversally at a singleton. The same property is shared by any $(2m -1)$-sphere in $\C^m_{\mathsf{coef}}$ that is centered on the origin $\mathbf 0$. Therefore, with the help of $\Psi_t$, $\d V(\mathsf{Sym}^mB^2)$ is homeomorphic to a $(2m - 1)$-sphere, and $V(\mathsf{Sym}^mB^2)$ to a closed ball $B^{2m}$. Hence $\mathsf{Sym}^mB^2 \approx B^{2m}$ topologically.
\smallskip

A similar argument helps to analyze the topology of $\mathsf{Sym}^m(\H)$. 

The boundary of the closed half-ball $B^2_+:=  B^2 \cap \H$  consists of the segment $I = [-1, 1] \subset \R$ and the arc $A :=  \{z \in \C |\; |z| =1,\, \mathsf {Im}(z) \geq 0\}$. 

Consider the set $\mathcal S$ of divisors $D \in  \mathsf{Sym}^m\H$ such that their support $\D \subset B^2_+$ and $\D \cap A \neq \emptyset$.  Again, there is a unique origin-centered dilation $t: \C \to \C$ such that $t(\D) \cap A \neq \emptyset$ and $t(\D) \subset B^2_+$, provided that $D \neq m\{0\}$. Therefore each $\psi_t$-trajectory through a  point $D \in \mathsf{Sym}^m\H$, except for the trivial trajectory through $m\{0\}$, has a unique point of  intersection with $\mathcal S$.

The boundary $\d(\mathsf{Sym}^m\H)$ in $\mathsf{Sym}^m\C$ consists of divisors $D$ such that their support $\D$ has a nonempty intersection with the real line $\R \subset \H$. Evidently, $\d(\mathsf{Sym}^m\H)$ is invariant under $\psi_t$-flow. Thus, $\d(\mathsf{Sym}^m\H)$ topologically is a cone over $\d(\mathsf{Sym}^m\H) \cap \mathcal S$, the set of divisors whose support $\D \subset B^2_+$ and such that $\D \cap I \neq \emptyset$.



Now, take a point $p \in \d B^2$ and a small \emph{open} $\e$-ball $B_\e(p) \subset \C$ with the center at $p$. Since $ H:= B_\e(p) \cap B^2$ is homeomorphic to the half-plane $\H$, $\mathsf{Sym}^m\H \approx \mathsf{Sym}^m H$. 

Consider $m\cdot p$ as a point in $\d(\mathsf{Sym}^mB^2) \approx \d B^{2m}$. Then $\mathsf{Sym}^m H$ can be viewed as an open $\e$-neighborhood of $m\cdot p$ in the space $\mathsf{Sym}^mB^2 \approx B^{2m}$,  the distance in the closed ball $\mathsf{Sym}^mB^2$ being the Hausdorff distance between $\mathsf{S}_m$-orbits in $(B^2)^m \subset \C^m$. At least for a sufficiently small $\e > 0$, that neighborhood $\mathsf{Sym}^mH$, being an $\e$-neighborhood of a point $m\cdot p \in \d(B^{2m})$,  is homeomorphic to $\R^{2m}_+$, and  so is $\mathsf{Sym}^m\H$.
\end{proof} 

For $\omega \in \Omega_{\langle d]}$, consider the space 
\begin{eqnarray}\label{eq4.1} 
e_\omega := \mathsf{Sym}^{|\sup(\omega)|}\R\, \times \, \mathsf{Sym}^{\frac{d - |\omega|}{2}}\H.
\end{eqnarray}
Denote by $\s_\omega$ the subset of $e_\omega$ that consists of divisor pairs $D' \times D''$ such that $\sup(D') \cup \sup(D'') \subset B^2$, but not of its interior. 

Let the subset $e_{\a, \omega} \subset e_\omega$ be defined by the constraint: the divisor $D' + D'' + \tau(D'')$ is $\a$-balanced. Similarly, let $\s_{\a, \omega} \subset \s_\omega$ be defined by the property of  $D' + D'' + \tau(D'')$ being $\a$-balanced.
\smallskip

In what follows, by a ``cell complex"  we mean a Hausdorff topological space $Z$ which is a disjoint union of its subsets $\{e^\circ_\a\}_\a$, where each $e^\circ_\a$ is homeomorphic to an open ball $B^\circ_\a$ of some dimension $d_\a$. Let $B_\a$ be either a closed ball, or an open ball together with a northern hemisphere in  its spherical boundary\footnote{This deviation from the standard definition of $CW$-complex is due to our need to consider germs of classical $CW$-complexes.}. The homeomorphism $B^\circ_\a \to e^\circ_\a$ must extend to a continuous map $\phi_\a: B_\a \to Z$, so that the image $\phi_\a(\d B_\a)$ is contained in a union of finitely many cells $e_\b$ of dimensions lower than $d_\a$. By definition,   $Y \subset Z$ is closed if, for each $\a$, $\phi_\a^{-1}(Y \cap \phi_\a(B_\a) )$ is closed in $B_\a$. Note that $e_\a$, the closure of $e^\circ_\a$ in $Z$, coincides with  $\phi_\a(B_\a)$. We do not require that $e_\a$ will be homeomorphic to a closed ball. 
\smallskip

The next proposition describes one particularly important cellular structure on the space $\mathcal P^d \approx \R^d_{\mathsf{coef}}$ of real degree $d$ monic polynomials, the structure induced by the combinatorial types $\omega$ of their real divisors.

\begin{theorem}\label{th4.1}  Let $\omega \in \Omega_{\langle d]}$. Then the following structures are available:
\begin{itemize}
\item each pure stratum  $\mathsf R_\omega \subset \R^d_{\mathsf{coef}}$ is homeomorphic to an open ball\footnote{In general, the intersection of $\mathsf R_\omega$ with a ball $B^d \subset \R^d$, centered at the origin, topologically is not a ball.} of codimension $|\omega|'$. 
\item the space $\s_\omega$ is homeomorphic to a closed $(d - |\omega|' -1)$-ball, and the space $\s_{\a, \omega}$  to a closed $(d - |\omega|' - 2)$-ball.  The space $e_\omega$ is a positive cone over $\s_\omega$, and $e_{\a, \omega}$ is a positive cone over $\s_{\a, \omega}$.
\item for each $\hat \omega \in \Omega_{\langle d]}$, the strata $\{\mathsf R_\omega\}_{\omega \preceq \hat \omega}$, define the structure of a cell complex on the real affine variety  $\mathsf R_{\hat\omega_\succeq} \subset \R^d_{\mathsf{coef}}$. The attaching maps $$\Phi_\omega^\d : \d e_\omega \to \mathsf R_{\omega_\succeq} \setminus \mathsf R_\omega$$ for the cells $e_\omega$ are described in formulas (\ref{eq4.4})\footnote{See the proof of this theorem for the constructions of the relevant maps.}.
\item the space $\R^d_{\mathsf{coef}}$ admits  a $1$-parameter flow $\Psi_t$  which has a single stationary point $\mathbf 0$ (a source), is transversal to each sphere $S^{d - 1}_{\mathsf{coef}} \subset \R^d_{\mathsf{coef}}$,  centered on $\mathbf 0$, and preserves each stratum $\mathsf R_\omega$. Thus, the cellular structure on $\R^d_{\mathsf{coef}}$ (described in the third bullet)  is a cone over a similar cellular structure on $S^{d - 1}_{\mathsf{coef}}$. 
\end{itemize}
\end{theorem}

\begin{proof} Any divisor in $\R$ comes with a particular linear order among the points in its support. Let $P$ be a real polynomial of degree $d$. Its real divisor $D_\R(P)$ of the combinatorial type $\omega$ is determined (with the help of $\omega$) by the support $\sup(D_\R(P))$, a point the chamber $\Pi^p_\circ$ of $\R^p = \{(y_1, \dots , y_p)\}$, defined by the  inequalities $y_1 < y_2 < \dots < y_p$, where $p = |\sup(\omega)|$. In fact, $\Pi^p_\circ$ is one of the $2^{p-1}$ chambers-cells in which the hyperplanes $\{y_i = y_{i+1}\}_i$ divide $\R^p$. 

The set $\Pi^p_\circ$, admits an obvious embedding into the space $\R^{|\omega|}$ with the coordinates $(x_1, \dots , x_{|\omega|})$. It is defined there by the system of $\omega$-dependent equations and inequalities: $$x_1 = x_2 = \dots = x_{\omega(1)} < x_{\omega(1) +1} =  x_{\omega(1) + 2} = \dots =  x_{\omega(1) +\omega(2)} < \dots $$
We  denote by $\Pi_\omega^\circ$ the solution set of this system. Note that $\Pi_\omega^\circ$ is homeomorphic to the open chamber $\Pi^{|\sup(\omega)|}_\circ \subset \R^{|\sup(\omega)|}$.  Let  $\Pi_\omega$ be the closure of $\Pi_\omega^\circ$ in $\R^{|\omega|}$. That closure is homeomorphic to $\Pi^{|\sup(\omega)|} \approx \mathsf{Sym}^{|\sup(\omega)|}\R$. Indeed, each orbit of the natural action of the symmetric group $\mathsf S_{|\sup(\omega)|}$ on $\R^{|\sup(\omega)|}$ intersects the closed chamber $\Pi^{|\sup(\omega)|}$ at a singleton. 
\smallskip

With its real roots being fixed, a polynomial $P$ is determined by the unordered  pairs of its conjugate roots (possibly with multiplicities).  Then the conjugate (non-real) pairs can be identified with points of the space $\mathsf{Sym}^m\H^\circ$, where $m = \frac{1}{2}(d - |\omega|)$. Therefore, $(\mathsf{Sym}^d\C)_\omega^\tau$ is homeomorphic to the space $\Pi^p_\circ \times \mathsf{Sym}^m\H^\circ$, where $p := |\sup(\omega)|$. Since $\H^\circ$ is homeomorphic to $\C$ and $\mathsf{Sym}^m\C$ is homeomorphic to $\C^m$ via the Vi\`{e}te Map $V$, we conclude that $\mathsf{Sym}^m\H^\circ$  is homeomorphic to $\C^m$. Thus, $$(\mathsf{Sym}^d\C)_\omega^\tau \approx \Pi^p_\circ \times \mathsf{Sym}^m\H^\circ \approx  \Pi^p_\circ \times \C^m$$, an open cell of dimension $$p + 2m =  |\sup(\omega)| + (d - |\omega|) = d - |\omega|'.$$ Note that, the Vi\`{e}te map $V$ generates a smooth homeomorphism between the spaces $(\mathsf{Sym}^d\C)_\omega^\tau \approx \R^{d - |\omega|'}$ and  $\mathsf R_\omega := (\R^d_{\mathsf{coef}})_\omega$. Thus the claim of the first bullet is validated.
\smallskip

The flow $\Psi_t : \R^d_{\mathsf{coef}} \to \R^d_{\mathsf{coef}}$ (see the fourth bullet) is induced by the flow $\psi_t$ in the root space $(\mathsf{Sym}^d\C)^\tau$ that applies  dilatations by real factors $t > 0$ to  each $\tau$-symmetric root configuration in $\C$.  The transplantation of  $\psi_t$ to $\R^d_{\mathsf{coef}}$ is done via the  Vi\`{e}te Map $V$, a smooth homeomorphism  $(\mathsf{Sym}^d\C)^\tau \to \R^d_{\mathsf{coef}}$. Evidently, each stratum $(\mathsf{Sym}^d\C)_\omega^\tau$ is $\psi_t$-invariant; therefore each stratum $\mathsf R_\omega$ must be invariant under $\Psi_t$.

We leave to the reader to verify that, for any polynomial $P \neq z^d$, its $\Psi_t$-trajectory is transversal to the spheres $S^{d - 1}_{\mathsf{coef}} \subset \R^d_{\mathsf{coef}}$, centered on $\mathbf 0$ (the verification is a straightforward computation). This proves the claim in the forth bullet.
\smallskip

Next, we turn our attention to the second bullet of the theorem. Each divisor $D \in \mathsf{Sym}^m\H$ produces  to a unique divisor $D_\R$, the part of $D$ that is supported in $\R$. We view $D_\R$ as an element  of $\mathsf{Sym}^l\R$, $l \leq m$. For the majority of $D$'s, $D_\R$ will have an empty support, so we interpret $\mathsf{Sym}^0 \R$ as an empty set. 

Recall that $e_{\omega} := \Pi_\omega \times \mathsf{Sym}^m\H$, where $m = \frac{1}{2}(d - |\omega|)$ (see (\ref{eq4.1})).  In the proof of Lemma \ref{lem4.1}, we have seen that $e_{\omega}$ is homeomorphic to  $\Pi_\omega \times \R^{d - |\omega|}_+$---a $|\sup(\omega)|$-dimensional pyramid  times a half-space in $\R^{d - |\omega|}$. 

The boundary of the unit half-disk $B^2_+  := \{z \in \C :\; |z| \leq 1,\, \mathsf {Im}(z) \geq 0\}$ consists of the segment $I = [-1, 1] \subset \R$ and the arc $$A :=  \{z \in \C :\; |z| =1,\, \mathsf {Im}(z) \geq 0\}.$$ 

The multiplicative group $\R^\ast_+$ of positive real numbers acts semi-freely on the space $\Pi_\omega \times \mathsf{Sym}^m\H$ by the diagonal $\psi_t$-action which is induced by the origin-centered dilations in $\C$. Its only fixed point is the origin $\mathbf 0 := \{0\} \times \{0\}$. This $\R^\ast_+$-action on $e_\omega \setminus \mathbf 0$ admits a  compact section $\s_\omega$ that consists of points $D' \times  D'' \in e_\omega$ such that $\sup(D') \cup \sup(D'')$ is contained in the unit  half-disk $B^2_+  \subset \H$ (centered on the origin) and has a nonempty intersection with the  arc $A$. 

Thus  $\s_\omega$ is  the set of pairs $D' \in \mathsf{Sym}^p(I), D'' \in \mathsf{Sym}^m(B^2_+)$ such that either $\sup(D') \cap \d I  \neq \emptyset$, or $\sup(D'') \cap A \neq \emptyset$, or both. Here $p = |\sup(\omega)|$. Note that if $\sup(D') \cap \d I \neq \emptyset$, then $\sup(D') \cap A \neq \emptyset$ since $\d A = \d I$. Therefore, $\s_\omega$ can be also described as a set of pairs $(D', D'')$ such that $\sup(D' + D'') \subset B^2_+$ and $\sup(D' + D'') \cap A \neq \emptyset$.

Recall that $\mathsf{Sym}^pI \approx \D^p$, a $p$-simplex. By Lemma \ref{lem4.1},  $\mathsf{Sym}^mB^2_+ \approx B^{2m}_+$, a half-ball. Let  $\delta B^{2m}_+ \subset \d(B^{2m}_+)$ denote the northern hemisphere in the boundary of the ball  $B^{2m} \supset B^{2m}_+$. It corresponds to the divisors $D''$ with the property $\sup(D'') \cap A \neq \emptyset$.

In new notations, the section $\s_\omega$ is  the set of pairs $D' \in \D^p, D'' \in  B^{2m}_+$ such that either $D' \in \d\D^p$, or $D'' \in \delta B^{2m}_+$, or both.  As a result, we get a homeomorphism
$$\s_\omega \approx  \big(\d\D^p \times B^{2m}_+\big) \cup_{\d\D^p \times \delta B^{2m}_+} \big(\D^p \times \delta B^{2m}_+\big)$$
whose target topologically is a $(d - |\omega|' -1)$-ball. Indeed, as the formula above testifies, $\s_\omega$ is obtained from the sold torus $T := S^{p-1}\times D^{2m}$ by a attaching a $p$-handle along $S^{p-1}\times D^{2m-1} \subset \d T$.  

Therefore $e_\omega$, an infinite positive cone over the ball $\s_\omega$, is homeomorphic to a positive cone in $\R^{d - |\omega|'}$ with a closed $(d - |\omega|' -1)$-ball base .
\smallskip

Now we proceed  to describe the attaching maps (see (\ref{eq4.3}) and (\ref{eq4.4})) for the cells $e_\omega$. With this  goal in mind, we will ``partially compactify" the pure stratum $\mathsf R_\omega$ in order to form an ``honest" $(d - |\omega|')$-cell $e_\omega$ and  will show how to attach its boundary $\d e_\omega$ to the strata $\{\mathsf R_{\tilde\omega}\}_{\tilde\omega \in \omega_\succ}$ of  dimensions smaller than $\dim(\mathsf R_\omega)$. This cell $e_\omega$ can be regarded as a \emph{resolution} of the real variety $\mathsf R_{\omega_\succeq}$.


Consider the maps
\begin{eqnarray}\label{eq4.2}
\Theta_\omega:\,  e_\omega := \Pi_\omega \times \mathsf{Sym}^{\frac{d - |\omega|}{2}}\H  \longrightarrow (\mathsf{Sym}^d\C)^\tau
\end{eqnarray}
defined by the formula $\Theta_\omega(D' \times D'') = D' + D'' + \tau(D'')$, where $D' \in  \Pi_\omega$, $D'' \in \mathsf{Sym}^{\frac{d - |\omega|}{2}}\H$, $\tau(D'')$ stands for the complex conjugate of the divisor $D''$, and "$+$" denotes the sum of divisors in $\C$. 

As formula (\ref{eq4.2}) testifies,  each map $\Theta_{\omega}$ doubles the real part $D''_\R$ of each divisor $D'' \in \mathsf{Sym}^m\H$ and thus mimics the merging of conjugate pairs of complex roots into the appropriate real roots of even multiplicity. 





Note that the restriction of $\Theta_\omega$ to $e_\omega^\circ :=  \Pi_\omega^\circ \times \mathsf{Sym}^{\frac{d - |\omega|}{2}}\H^\circ$,  the interior of $e_\omega$,  is a homeomorphism onto the pure stratum $(\mathsf{Sym}^d\C)_\omega^\tau$.

The restriction $\Theta_\omega^\d$ of $\Theta_\omega$  to  the boundary $\d e_\omega $ provides us with the attaching maps
\begin{eqnarray}\label{eq4.3}
\big\{\Theta_\omega^\d:  \d e_\omega  \to (\mathsf{Sym}^d\C)^\tau\big\}_{\omega \in \Omega_{\langle d]}}. 
\end{eqnarray}
By the very construction of $e_\omega$, the $\Theta_\omega$-image of $\d e_\omega$ belongs to the union of strata  $\big\{(\mathsf{Sym}^d\C)_{\tilde\omega}^\tau \big\}_{\tilde\omega}$,  where $\tilde\omega \in \omega_\succeq$. Indeed, if $D' \times D'' \in  \d e_\omega$, then either $D' \in \d\Pi_\omega$, or $D'' \in \d(\mathsf{Sym}^{\frac{d - |\omega|}{2}}\H)$. In the first case, $D'$ is obtained from some $D \in \Pi_\omega^\circ$ via  merge operations;  thus $(D' + D'' + \tau(D''))_\R$ is obtained from $(D + D'' + \tau(D''))_\R$ by the same merges.  In the second case, $(D' + D'' + \tau(D''))_\R$ can be obtained from $D'$ by inserting $(D'' + \tau(D'')_\R$.

Therefore the maps $\{\Theta_\omega^\d\}$ from (\ref{eq4.3}) define cellular structures on the real affine variety $(\mathsf{Sym}^d\C)^\tau$ and its subvarieties $\{(\mathsf{Sym}^d\C)_{\hat\omega\succeq}^\tau\}_{\hat\omega}$. 
\smallskip

We notice that all the maps $\Theta_\omega$ are equivariant under the $\R^\ast_+$-actions $\psi_t$ and $\Psi_t$ in the source and target spaces, respectively,  so that the attaching maps are consistent with the cone structures  of $e_\omega$ and of  the strata $\{(\mathsf{Sym}^d\C)_{\hat\omega}^\tau\}$. Moreover, the sections $\s_\omega$ are mapped by  $\Theta_\omega$ to some sections $\mathsf S_\omega \subset \mathsf R_\omega$ of the flow $\Psi_t$. The space $\mathsf S_\omega$ is defined  as the set of real polynomials with all their roots residing in the ball $B^2 \subset \C$, but not in its interior, and with the combinatorics of the real roots being prescribed by the poset $\omega_\succeq \subset \Omega_{\langle d]}$.

With the help of the Vi\`{e}te homeomorphism $V$, the maps
\begin{eqnarray}\label{eq4.4}
\{\Phi_\omega : e_\omega  \stackrel{\Theta_\omega}{\longrightarrow} (\mathsf{Sym}^d\C)^\tau \stackrel{V}{\longrightarrow} \mathsf R_{\omega_\succeq} \subset \R^d_{\mathsf{coef}}\}_{\omega \in \Omega_{\langle d]}}\\ \nonumber
\{\Phi_\omega^\d :  \d e_\omega  \stackrel{\Theta_\omega^\d}{\longrightarrow} (\mathsf{Sym}^d\C)^\tau \stackrel{V}{\longrightarrow} \mathsf R_{\omega_\succ} \subset \R^d_{\mathsf{coef}}\}_{\omega \in \Omega_{\langle d]}}
\end{eqnarray}
define cellular structures in $\R^d_{\mathsf{coef}}$ and its subvarieties $\{\mathsf R_{\hat\omega_\succeq}\}_{\hat\omega}$. Again, the attaching maps $\{\Phi_\omega, \Phi_\omega^\d\}_\omega$ are $\R^\ast_+$-equivariant. They consistent with the cone structures in the strata $e_{\hat\omega}$ and $\mathsf R_{\hat\omega}$. In particular, we get the maps:
\begin{eqnarray}\label{eq4.5}
\{\Phi_\omega : \s_\omega  \stackrel{\Theta_\omega}{\longrightarrow} (\mathsf{Sym}^d\C)^\tau \stackrel{V}{\longrightarrow} \mathsf S_{\omega_\succeq} \subset S^{d -1}_{\mathsf{coef}}\}_{\omega \in \Omega_{\langle d]}}\\ \nonumber
\{\Phi_\omega^\d :  \d \s_\omega  \stackrel{\Theta_\omega^\d}{\longrightarrow} (\mathsf{Sym}^d\C)^\tau \stackrel{V}{\longrightarrow} \mathsf S_{\omega_\succ} \subset S^{d -1}_{\mathsf{coef}}\}_{\omega \in \Omega_{\langle d]}}
\end{eqnarray}
Here $S^{d -1}_{\mathsf{coef}}$ denotes the space of  degree $d$ real monic polynomials whose roots are in $B^2$, but not in its interior. With the help of $\Psi_t$, the space $S^{d -1}_{\mathsf{coef}}$ is diffeomorphic to the standard sphere $S^{d -1}$. This completes the proof of the third bullet .
\end{proof}

Results, similar to the ones described  in Theorem \ref{th4.1},  are valid for the space of real degree $d$ monic polynomials with a fixed coefficient $d\cdot\a$ of $z^{d - 1}$.

\begin{theorem}\label{th4.2} Let $\omega \in \Omega_{\langle d]}$ and $\a \in \R$. We denote by $\omega_\star: 1 \to d$ the minimal element of the poset $\Omega_{\langle d]}$. The following properties hold:
\begin{itemize}
\item the stratum $\mathsf R_{\a, \omega}\subset \R^{d -1}_{\mathsf{coef}, \a}$ is an open ball of codimension $|\omega|'$, 
\item for each $\hat \omega \in \Omega_{\langle d]}$, the strata $\{\mathsf R_{\a, \omega}\}_{\omega \preceq \hat \omega}$ give rise to a cellular structure on the affine variety $\mathsf R_{\a, \hat\omega_\succeq} \subset \R^{d -1}_{\mathsf{coef}, \a}$. The attaching maps $$\Phi_{\a,\omega}^\d : \d e_{\a,\omega} \to \mathsf R_{\a, \omega_\succeq} \setminus \mathsf R_{\a, \omega}$$ for the cells $e_{\a, \omega}$ are described by formulas similar to formulas  (\ref{eq4.4})-(\ref{eq4.5}),
\item the space $\R^{d -1}_{\mathsf{coef}, \a}$ admits  a $1$-parameter flow $\Psi_t^\a$ which has a single stationary point $O_\a$ (a source), is transversal to each sphere $S^{d - 2}_{\mathsf{coef}, \a}$, centered on $O_\a$, and preserves each stratum $\mathsf R_{\a, \omega}$. Thus, the $\Omega_{\langle d]}$-labeled cellular structure on $\R^{d - 1}_{\mathsf{coef}, \a}$ is a cone over a similar  $(\Omega_{\langle d]} \setminus \omega_\star)$-labelled cellular structure on the sphere $S^{d - 2}_{\mathsf{coef}, \a}$.
\end{itemize}
\end{theorem}

\begin{proof} We turn to the $\a$-balanced divisors, which tell a similar story. 

Consider a flow $\phi_t: \{z_1, \dots , z_d\} \to \{z_1- \frac{\a}{d}t,\, \dots ,\, z_d - \frac{\a}{d}t\}$ in $(\mathsf{Sym}^d\C)^\tau$. For $t =1$, it maps $\{z_1, \dots , z_d\}$ to $\{z_1- \frac{\a}{d}, \dots , z_d- \frac{\a}{d}\}$, a $\a$-balanced configuration. Note that $\phi_t$ preserves the $\omega$-stratification in $(\mathsf{Sym}^d\C)^\tau$. Hence, with the help of the Vi\`{e}te map $V$, $\phi_t$ gives rise to a retraction of the stratum $\mathsf R_\omega$ onto the stratum $\mathsf R_{\a, \omega}$. Therefore, $\mathsf R_{\a, \omega}$ is an open ball of dimension $d- |\omega|' -1$ and codimension  $|\omega|'$ in $\R_{\mathsf{coef}, \a}^d$.

Next, consider the flow $\Psi_t^\a : \R^d_{\mathsf{coef}, \a} \to \R^d_{\mathsf{coef}, \a}$, induced with the help of the Vi\`{e}te map $V$ by a flow $\{\psi_t^\a\}_{t > 0}$ in the root space $(\mathsf{Sym}^d_\a\C)^\tau$. The flow  $\psi_t^\a$ applies $t$-dilatations that are centered on the point $o_\a := \a/d \in \R$  to each $\tau$-symmetric and $\a$-ballanced divisor $D$ in $\C$. The dilatation $\psi_t^\a$ keeps the center of gravity of the weighted configuration $\psi_t^\a(D)$ at the point $o_\a \in \C$, so that  $\psi_t^\a(D)$ is $\a$-balanced for all $t$. Evidently, these dilatations preserve the $\omega$-type of each $\a$-balanced configuration in $\C$. As a result, each $\mathsf R_{\a, \omega}$ is $\Psi_t^\a$-invariant. 

By a direct computation, any $\Psi_t^\a$-trajectory is transversal to the spheres $S^{d - 2}_{\mathsf{coef}, \a} \subset \R^d_{\mathsf{coef}, \a}$ centered on the point $O_\a := V(o_\a, \dots, o_\a)$. The exception is the trivial trajectory through $O_\a$. 
\smallskip

We denote by $e_{\a, \omega}$ the subset of $e_{\omega}$ (see (\ref{eq4.1})) that consists of pairs $(D', D'')$, where $D' \in \Pi_\omega$, $D'' \in \mathsf{Sym}^m\H$, and such that the divisor $$D' + D'' + \tau(D'')$$ is $\a$-balanced. In other words, if $D' = \sum_i m_i x_i$ and $D'' = \sum_k m_k z_k$, then $e_{\a, \omega}$ is the space of a fibration  over the base $\mathsf{Sym}^m\H \approx \R_+^{2m}$ with the cell-like fiber $F_{D''} \subset \Pi_\omega$ (over the point $D''$) that is defined by the constraint $$\sum_i m_i x_i = d\cdot\a - \sum_k 2m_k \mathsf{Re}(z_k).$$ 

Denote by  $\s_{\a, \omega}$ the subset of $e_{\a, \omega}$ that consists of $\a$-balanced pairs $(D', D'')$ such that $\sup(D') \cup \sup(D'')$ is contained in the $\a$-centered unit ball $B^2_+(\a)$, but not in its interior.  As in the ``unbalanced" case, $\s_{\a, \omega}$ is a section of the $\psi^\a_t$-flow in $e_{\a, \omega}$. This results in $e_{\a, \omega}$ being a positive cone over the base $\s_{\a, \omega}$. By an argument as in the proof of Theorem \ref{th4.1}, $\s_{\a, \omega}$ is homeomorphic to a closed $(d - 2 -|\omega|')$-ball.

For any real $\a$,  the maps $\Theta_{\omega}$ in (\ref{eq4.2}), being restricted to $e_{\a, \omega} \subset e_{\omega}$, produce the maps $\Theta_{\a, \omega} : e_{\a, \omega} \to (\mathsf{Sym}^d_\a\C)^\tau$ which, with the help of $V$, give rise to the maps 
$$\Phi_{\a, \omega} : e_{\a, \omega} \to \mathsf R_{\a, \omega_\succeq} \subset \R_{\mathsf{coef}, \a}^{d -1}$$ 
, already familiar from the ``unbalanced" formulas (\ref{eq4.4}). 
They  define a cellular structure on the real variety $\R_{\mathsf{coef}, \a}^{d -1}$ and its subvarieties $\mathsf R_{\a, \tilde\omega_\succeq}$. Analogously, the maps 
$$\Phi_{\a, \omega} : \s_{\a, \omega} \to \mathsf S_{\a, \omega_\succeq} \subset S_{\mathsf{coef}, \a}^{d - 2}$$  define a cellular structure on the sphere $S_{\mathsf{coef}, \a}^{d -2}$ and its strata $\mathsf S_{\a, \tilde\omega_\succeq}$. Here $S_{\mathsf{coef}, \a}^{d -2}$ is the space of  real monic polynomials $P(z)$ of the form  $$z^d + \a z^{d -1} + \dots$$  such that the $P$-roots are contained in $B^2(\a)$, but not in its interior. Note that $P(z) \neq (z - \frac{\a}{d})^d$---the apex of the cone. Similarly, $\mathsf S_{\a, \omega_\succeq} \subset S_{\mathsf{coef}, \a}^{d -2}$,  is the set of such polynomials $P(z)$ whose real roots have the combinatorics that is prescribed by the poset $\omega_\succeq$. 
\end{proof}

\noindent{\bf Example 4.1.} Let $\omega_\star$ be the minimal element of the poset $\Omega_{\langle d]}$. Theorem \ref{th4.1} claims that the sphere $S^{d - 1}$ admits a cellular structure whose sells $\{\mathsf S_\omega\}$ of codimension $|\omega|'$ are indexed by the elements of the poset $(\Omega_{\langle d]} \setminus \omega_\star,\, \succ)$.  Moreover, $\mathsf S_{\omega'_\succeq} \subset \mathsf S_{\omega''_\succeq}$ if and only if $\omega'' \succeq \omega'$. In other words, the poset $(\Omega_{\langle d]}  \setminus \omega_\star,\, \succ)$ provides a complete set of instructions for  assembling $S^{d - 1}$! You may glance at Fig. 3 to examine how the assembly works for $S^2$. Note that this cellular structure is not a regular one: the closures $\mathsf S_{\omega_\succeq}$ of cells $\mathsf S_\omega$ are not necessarily closed balls. \hfill\qed
\smallskip

\smallskip 

The next lemma describes the attaching maps $\Phi_\omega^\d:  \d e_\omega \to R^d_{\mathsf{coef}}$ for the cells $e_\omega$ as being finitely ramified over their images. The degrees of ramification are described by combinatorics-flavored numbers  $\{o(\omega, \tilde\omega)\}$ whose exact Definition \ref{def4.1} will be provided below. 

\begin{lemma}\label{lem4.2} Each map $\Phi_\omega:  e_\omega \to \R^d_{\mathsf{coef}}$ from (\ref{eq4.4})  has finite fibers, and is a bijection on the interior of the cell $e_\omega$. 

For any point-polynomial  $Q \in \Phi_\omega(\d e_\omega)$, the cardinality of the fiber  $\Phi_\omega^{-1}(Q)$ is equal to the number $o(\omega, \tilde\omega)$ introduced in Definition \ref{def4.1} below. Here $\tilde\omega$ denotes the combinatorial pattern of the real divisor $D_\R(Q)$.
\end{lemma}

\begin{proof} Let us fix a degree $d$ conjugation-invariant divisor $$D = \Theta_\omega((D', D'')) := D' + D'' + \tau(D'')$$, where $D' \in \Pi_\omega$,  $D'' \in \mathsf{Sym}^m\H$, and $m := (d- |\omega|)/2$.  We denote by $\tilde\omega$ the combinatorial type of $D_\R$. 
Then there exist only finitely many pairs $(D', D'')$ that deliver $D$. Indeed, since $\sup(D') \subset \R$, the divisor $D'' - D''_\R$ with the support in $\H^\circ$ is uniquely determined by $D$.  Note that $D_\R$ and $D'$ (with the support in $\R$) differ by $2D''_\R$. This leaves only a finite set of choices for $D' = D_\R -  2D''_\R$, whose support must be contained in $\sup(D_\R)$ and whose degree is bounded by $\deg(D_\R)$. 
\smallskip

If $(D', D'') \in \d e_\omega$, then either $D' \in \d\Pi_\omega$, or $D'' \in \d(\mathsf{Sym}^m\H)$ (i.e. $\sup(D'') \cap \R \neq \emptyset$), or both. In the first case, $\omega(D')$, the combinatorial type of $D'$, is obtained from $\omega$ by a sequence of merge operations $\{\mathsf M_j\}$. In the second case, $\tilde\omega:= \omega(D_\R)$, the combinatorial type of the $\Theta_\omega$-image of $(D', D'')$,  is obtained from $\omega$ by a sequence of insert operations $\{\mathsf I_k\}$. In the third mixed case, one applies both types of operations.

As a result, the combinatorics of reconstructing $(D', D'')$ from  $D = D' + D'' + \tau(D'')$ can be described as follows. If $\tilde \omega$ is the combinatorial pattern of $D_\R$, then $\omega'$, the combinatorial pattern of $D' = D'_\R$ can be obtained by: (1) subtracting from  $\tilde\omega$ a non-negative function $2\omega''$, such that  $\tilde\omega - 2\omega''$ is again a non-negative function ($\omega''$ represents the divisor $D''_\R$), and (2) deleting all the positions $i$ where $(\tilde\omega - 2\omega'')(i) = 0$. 

Let us denote by $K$ this ``repackaging" operator from (2); it takes any histogram and deletes from it all the columns of zero height.    

Since $D' \in \d\Pi_\omega$,  we conclude that  $\omega' = K(\tilde\omega - 2\omega'')$ should be obtainable from $\omega$ via a sequence of merge operations alone. 

For each $\omega \in \Omega$, let us denote by $\omega_{\succeq\mathsf M}$ the subset of $\Omega$ that consists of elements that can be obtained from $\omega$ by the merge operations $\{\mathsf M_j\}$ alone. 

In these notations, we get that $K(\tilde\omega - 2\omega'') \in \omega_{\succeq\mathsf M}$.

\begin{definition}\label{def4.1} For any pair $\omega \succeq \tilde\omega$ in $\Omega$, consider the subset $O(\omega, \tilde\omega) \subset \Omega$ such 
that, for any $\omega'' \in O(\omega, \tilde\omega)$, the following properties hold:
\begin{itemize}
\item  the function $\tilde\omega - 2\omega'' \geq 0$
\item $K(\tilde\omega - 2\omega'') \in \omega_{\succeq\mathsf M}$
\end{itemize}
Let $o(\omega, \tilde\omega)$  denote the cardinality of $O(\omega, \tilde\omega)$. \hfill\qed
\end{definition}


Recall that the Vi\`{e}te map $V: (\mathsf{Sym}^d\C)^\tau \to \R^d_{\mathsf{coef}}$ is a stratification-preserving homeomorphism of the $\Omega_d$-stratified spaces; in particular, it is bijective. Therefore the cardinality of each fiber $\Phi_\omega^{-1}(Q)$, where $Q \in  \R^d_{\mathsf{coef}}$, is equal to the cardinality of the fiber $\Theta_\omega^{-1}(V^{-1}(Q))$.

Pick $Q \in \Phi_\omega(\d e_\omega)$ and let  $\tilde\omega$ denote the combinatorial type of  the divisor $D_\R(Q)$. Then, by the considerations above,  we have proved that $| \Phi_\omega^{-1}(Q)| = | \Theta_\omega^{-1}(V^{-1}(Q))| = o(\omega, \tilde\omega)$. 
\end{proof}
\smallskip



For each $\omega \in \Omega_{\langle d]}$, let $\omega_{\prec \leadsto k}$ be the set of elements  $\tilde\omega \in \Omega_{\langle d]} $ that can produce $\omega$ as a result of a sequence of $k$ elementary operations $\{\mathsf M_i, \mathsf I_j\}$ being applied to $\tilde\omega$  (each such $\tilde\omega$ is the maximal element in a chain of $\omega$-predecessors in $\Omega_{\langle d]} $ of length $k$). Similar, let $\omega_{\preceq \leadsto k}$ be the set of elements in $\Omega_{\langle d]} $ that can produce $\omega$ by a sequence of $k$ elementary operations at most. 
\smallskip


Next lemma is an extension of Theorem \ref{th4.1} and Theorem \ref{th4.2}. It spells out the local arrangement of cells in the direction normal to a typical  pure stratum $\mathsf R_\omega$ or $\mathsf R_{ \a, \omega}$ in $\R_{\mathsf{coef}}^d$ or in $\R_{\mathsf{coef}, \a}^{d - 1}$, respectively.

\begin{lemma}\label{lem4.3} Let $\omega \in \Omega_{\langle d]}$ and $P \in \mathsf R_\omega$. We denote by $St_\nu(P)$ the star, normal in $\R^d_{\mathsf{coef}}$ to the pure stratum $\mathsf R_\omega$  at the point $P$. This normal star is homeomorphic to the space $\R^{|\omega|'}$ and inherits the structure of a cell complex from $\R^d_{\mathsf{coef}}$.  The cells $f_{\tilde\omega} :=  St_\nu(P) \cap \mathsf R_{\tilde\omega}$ of dimension $|\omega|' - |\tilde\omega|' $ are indexed by the elements $\tilde\omega \succeq \omega$. \smallskip

The incidence of cells $\{f_{\tilde\omega}\}$ in $St_\nu(P)$ is prescribed by the partial order in the poset $\omega_\preceq$: specifically, any element $\tilde\omega \in \omega_{\prec \leadsto k}$ gives rise to a cell $f_{\tilde\omega} \subset St_\nu(P)$ of dimension $k$.  It is contained in exactly $|\tilde\omega|' + \#(\tilde\omega^{-1}(2))$
cells $\{f_{\tilde\omega'}\}$ of dimension $k + 1$. The total number of  cells $f_{\tilde\omega} \subset St_\nu(P)$ of dimension $k$ is the cardinality of the set $\omega_{\prec \leadsto  k}$. \smallskip

Similar properties hold for the strata $\{St_\nu(P) \cap\mathsf R_{\a, \tilde\omega}\}_{\tilde\omega}$ in the  $\a$-balanced polynomial space $\R^d_{\mathsf{coef}, \a}$. 
\end{lemma}

\begin{proof} 
Let  $St(\mathsf R_\omega)$ be  the union of all cells $\{\mathsf R_{\tilde\omega}\}$ in $\R^d_{\mathsf{coef}}$ such that  $\mathsf R_{\tilde\omega_\succeq} \supset \mathsf R_\omega$. In contrast with the standard definition of the star, here we ignore the cells $\mathsf R_{\tilde\omega}$ such that  $\mathsf R_{\tilde\omega_\succeq} \cap \mathsf R_\omega \neq \emptyset$, but $\mathsf R_{\tilde\omega_\succeq}$ does not contain $\mathsf R_\omega$.  

Recall that any elementary operation $\mathsf M_j$ or $\mathsf I_j$ from (\ref{eq2.1})-(\ref{eq2.3}), being applied to an element $\omega \in \Omega$,  lowers its reduced norm $|\omega|'$ by $1$. Thus  each $\mathsf R_{\tilde\omega}$  from $St(\mathsf R_\omega)$ has $\textup{codim}(\mathsf R_\omega, \mathsf R_{\tilde\omega_\succeq}) = k$, if and only if, $\tilde\omega \in  \omega_{\prec  \leadsto k}$, that is,  if $\tilde\omega$ can be obtained from $\omega$ by a sequence of $k$ elementary resolutions $\{\mathsf M_j^{-1}\}$\footnote{These mimic a bifurcation of a real multiple root in a pair of real roots of the same combined multiplicity.} and elementary reductions by two $\{\mathsf I_j^{-1}\}$\footnote{These mimic the resolution of a real root $\a$ of multiplicity $2$ into two simple complex-conjugate roots.}.  
In particular, there are exactly as many cells $\mathsf R_{\tilde\omega}$ of dimension $\dim(\mathsf R_\omega) + 1$ as there are elementary resolutions and reductions of $\omega$.  Each value $\omega(i)$ can be "resolved" in $\omega(i) - 1$ distinct ways: $$(\omega(i)-1, 1),\, (\omega(i)-2, 2),\, \dots ,\,  (1, \omega(i) -1)$$ (the order does matter!), and each $\omega(i) = 2$ can be ``deleted" or ``reduced". Therefore, the total number of  elementary resolutions is $\sum_i (\omega(i) -1) = |\omega|'$.  The total number of reductions is $\#(\omega^{-1}(2))$, the cardinality of $\omega^{-1}(2)$. All together, there are $|\omega|' + \#(\omega^{-1}(2))$ elementary operations applicable to $\omega$.  Stated differently, the \emph{multiplicity} of each cell $\mathsf R_\omega$---the number of $(\dim(\mathsf R_\omega) + 1)$-cells containing it---is  the \emph{codimension} of $\mathsf R_\omega$ in $\R^d_{\mathsf{coef}}$ (or in $\R^{d-1}_{\mathsf{coef}, \a}$) plus  $\#(\omega^{-1}(2))$.
\smallskip

By Theorem \ref{th4.1}, $St_\nu(P)$, the germ  of  the vector space normal to $\mathsf R_\omega$ at $P$, has dimension $|\omega|'$. It is transversal at $P$ to each cell $\mathsf R_{\tilde\omega_\succeq}$ that contains $\mathsf R_\omega$. Thus, $St_\nu(P) \cap \mathsf R_{\tilde\omega_\succeq}$ is a germ-cell $f_{\tilde\omega}$ of dimension $|\omega|' - |\tilde\omega|'$.  The intersection  $\mathsf R_{\tilde\omega} \cap L_\nu(P)$ with the normal link $L_\nu(P)$---the boundary of $St_\nu(P)$--- is a cell $f_{\tilde\omega}^\d$ of dimension $|\omega|' - |\tilde\omega|' - 1$. This reveals the cellular structure of $\R^d_{\mathsf{coef}}$ and of $\R^{d-1}_{\mathsf{coef}, \a}$ at $P$ as a product of the cellular structure in $St_\nu(P)$ times the open cell $\mathsf R_\omega$. Therefore, the same combinatorics $(\Omega_{\langle d]}, \succ)$ governs both structures; in particular,  each cell $f_{\tilde\omega} \subset St_\nu(P)$ is contained in exactly $|\tilde\omega|' + \#(\tilde\omega^{-1}(2))$ cells of the next dimension. Similarly,  each cell $f_{\tilde\omega}^\d \subset L_\nu(P)$, $\tilde\omega \neq \omega$,  is contained in exactly $|\tilde\omega|' + \#(\tilde\omega^{-1}(2))$ cells of the next dimension.
\end{proof}

\noindent{\bf Remark 4.1.} Note that the combinatorics of these cellular structures is different from the combinatorics of the standard simplex $\D^d$: in a simplex, the multiplicity of each subsimplex is equal to its codimension (to $|\tilde\omega|'$ in our notations); thus the ``defect" $\#(\tilde\omega^{-1}(2))$  measures the deviation of the  $\Omega_{\langle d]}$-labeled cellular structures in  $\R^d_{\mathsf{coef}}$ from the standard simplicial one in $\D^d$. 
\hfill\qed
\section{On Spaces of Multi-tangent Trajectories}

As before, let $X$ be a compact smooth $(n+1)$-manifold with boundary. For  a  boundary generic field $v$, each $v$-trajectory $\g$ intersects the boundary $\d_1X$ at a finite number of points $a$ of multiplicities $m(a) \leq n+1$. Recall that formulas (\ref{eq1.2})-(\ref{eq1.4}) attach the multiplicity $m(\g)$, the reduced multiplicity $m'(\g)$, and the virtual multiplicity $\mu(\g)$ to each trajectory $\g$. For traversally generic fields, $m(\g) \leq 2(n+1)$,  $m'(\g) \leq n$ for every $\g$ (see Theorem 3.4 from \cite{K2}).
\smallskip 

The space of $v$-trajectories $\mathcal T(v)$ is given the quotient topology, so that the obvious map $\Gamma: X \to \mathcal T(v)$ is continuous. Let us stress again that if $v$ has singularities, the trajectory space $\mathcal T(v)$ is quite pathological (non-separable). In contrast, for nonsingular generic gradient fields $\mathcal T(v)$ is a decent space, a compact $CW$-complex!   We will prove this theorem in the next paper. 
\smallskip

For traversing fields $v$, all the fibers of $\Gamma$ are closed intervals or singletons. This property leads to 

\begin{theorem}\label{th5.1} For a traversally generic  field $v$ on $X$, the map $\Gamma: X \to \mathcal T(v)$ is a weak homotopy equivalence\footnote{In the next paper, we will show that $\Gamma$ actually is a homotopy equivalence for traversally generic fields.}. 

For any traversing field $v$ and any local coefficient system $\mathcal A$ on $X$, the map  $\Gamma$ is a homology equivalence: $$\Gamma^\ast: H^\ast(\mathcal T(v); \Gamma_\ast(\mathcal A)) \approx H^\ast(X; \mathcal A)$$ for all $\ast \geq 0$.
\end{theorem}

\begin{proof}  Suppose there exists an open cover $\mathcal U = \{U_\a\}_\a$ of  $\mathcal T(v)$, such that all the maps $\Gamma: \Gamma^{-1}(U_\a) \to U_\a$ are  weak homotopy equivalences.  Then, by Corollary 1.4 from  \cite{May},  $\Gamma: X \to \mathcal T(v)$ is a week homotopy equivalence, provided that the cover $\mathcal U$ is closed under finite intersections of its elements. 

So we need to construct the appropriate cover of $\mathcal T(v)$ and to prove that $\Gamma$ is a weak homotopy equivalence just locally.
\smallskip



By Lemma \ref{lem1.2}, for each $v$-trajectory $\g$, there exists a $\hat v$-adjusted neighborhood $\hat V_\g \subset \hat X$ of $\g$ such that, in special coordinates $(u, x, y)$,  $X$ is described by the polynomial inequality $\{P(u, x) \leq 0\}$, and the cylindrical neighborhood $\hat V_\g$ of $\g \subset X$ by the additional inequalities $\|x\| \leq \e$, $\|y\| \leq \e'$. 

Being the union of all $\hat v$-trajectories close to $\g$, this neighborhood $\hat V_\g$ determines a neighborhood $U_\g$ of the point $\g \in \mathcal T(v)$. Note that $\hat V_\g$ may have $\hat v$-trajectories that do not intersect $X$. Let us denote by $V_\g$ the subset of $\hat V_\g$ built out of $\hat v$-trajectories $\hat\g$ with the property $\hat\g \cap X \neq \emptyset$. By definition, the restriction $\Gamma :  V_\g \to U_\g$ is surjective.

Let us denote by $\omega_x \in \Omega$ the combinatorial pattern of the divisor $D_\R(P(u, x))$, where $x \in  \R^{|\omega_0|'}$ and $\omega_0$ is the combinatorial pattern of the divisor $D_\R(P(u, 0))$. If $x$ is such that $P(u, x) > 0$ for all $u$, then $\omega_x: \N_+ \to \Z_+$ is defined to be the zero map. The reduced norm $|\omega_x|'$ of such trivial $\omega_x$ is defined to be $-1$.

For each $k \in [0,\, |\omega_0|']$, consider the real subvariety $\mathcal X_{k} \subset \R^{|\omega_0|'}$ defined by the constraint $|\omega_x|' \geq k$. In other words, $x \in \mathcal X_{k}$ if and only if the reduced multiplicity $m'\big(D_\R(P(u, x))\big) \geq k$. In view of Theorem \ref{th4.1} and Theorem \ref{th4.2}, $\textup{codim}(\mathcal X_{k}, \R^{|\omega_0|'}) = k$. 

Let $\pi: V_\g \to \R^{|\omega_0|'} \times \R^{n - |\omega_0|'}$ denote the projection $(u, x, y) \to (x, y)$.

Put $$V_{\g, k} := \pi^{-1}(\mathcal X_{k} \times B_{\e'}(0) )$$, where $B_{\e'}(0)$ is the $\e'$-ball in $\R^{n- |\omega_0|'}$ with the center at $0$.  Thus $V_{\g, 0} := V_\g$, and $V_{\g, |\omega_0|'} = \g \times B_{\e'}(0)$. 

The map $\pi$ can be viewed as a composition $p\circ q$ of two maps: the quotient surjective map $q: V_\g \to U_\g \subset \mathcal T(v)$, whose fibers are closed segments, and a finitely ramified map $p: U_\g \to \R^{|\omega_0|'}$. The map $q$ is the restriction of the map $\Gamma: X \to \mathcal T(v)$ to the neighborhood $V_\g$.

Let $U_{\g, k} := q(V_{\g, k})$. We will argue by induction ``$k \Rightarrow k-1$". We claim that if $q: V_{\g, k}\cap X \to U_{\g, k}$ is a weak homotopy equivalence, then so is the map $q: V_{\g, k-1} \cap X \to U_{\g, k-1}$, provided $k > 0$.  

First we will show that the map $$\tilde q_{k-1}: (V_{\g, k-1}\cap X)/ (V_{\g, k} \cap X) \to U_{\g, k-1}/ U_{\g, k}$$ admits a continuous  section $\s_{k-1}$ such that $\s_{k-1}(U_{\g, k-1}/ U_{\g, k})$ is a deformation retract of $(V_{\g, k-1}\cap X)/ (V_{\g, k} \cap X)$. 


For each $x \in \mathcal X_{k -1}$, the set $\{P(u, x) \leq 0\}$ is a disjointed union of closed intervals $\{I_i(x)\}_i$. Let $u_i^\star(x)$ be the center of the interval  $I_i(x)$. With the help of $q$, the pair $(x, u_i^\star(x))$ determines the point $q(x, u_i^\star(x))$ in $\mathcal T(v)$.  Then we define the ``protosection"  $\s_{k-1}$ by the formula $$\s_{k-1}(q(x, u_i^\star(x)) := (x, u_i^\star(x)).$$
This formula is discontinuous for points  $q(x, u_i^\star(x)) \in U_{\g, k}$, where some intervals $$\{I_i(x)\}_{x \in \mathcal X_{k-1}\setminus \mathcal X_{k},\; i}$$ merge; however, it produces a continuous section $$\s_{k-1} : U_{\g, k-1}/ U_{\g, k} \to  (V_{\g, k-1}\cap X)/ (V_{\g, k} \cap X)$$ of the quotients. Now $ (V_{\g, k-1}\cap X)/ (V_{\g, k} \cap X)$ retracts on $\s_{k-1}(U_{\g, k-1}/ U_{\g, k})$ by collapsing each interval $I_i(x)$ on its center $u_i^\star(x)$. 

The basis of induction is $k = |\omega_0|'$. In this case, with the help of $q$,  $V_{\g, |\omega_0|'} := \g \times B_{\e'}(0)$ is homotopy equivalent to $U_{\g, |\omega_0|'} := B_{\e'}(0)$. 

By the inductive assumption, $q_k: V_{\g, k} \cap X \to U_{\g, k}$ is a weak homotopy equivalence. We have shown that $$\tilde q_{k-1}:   (V_{\g, k-1}\cap X)/ (V_{\g, k} \cap X) \to U_{\g, k-1}/ U_{\g, k}$$ is a homotopy equivalence. Therefore, comparing the exact homotopy sequences of the two triples, we conclude that $q_{k-1}: V_{\g, k-1}\cap X \to U_{\g, k-1}$ is a weak homotopy equivalence as well. In particular, it follows that   $q_0:= q: V_{\g, 0} \cap X \to U_{\g, 0}$ is a weak homotopy equivalence. 


By compactness of $X$, we can pick a finite $v$-adjusted cover $\mathcal V := \{V_\g\}$ of $X \subset \hat X$ and the corresponding cover $\mathcal U := \{U_\g:= \Gamma(V_\g)\}$ of $\mathcal T(v)$, so that each map $\Gamma: V_\g \to U_\g$ is a weak homotopy equivalence. Add to the list $\mathcal V$  all the multiple intersections $V_{\g_1} \cap V_{\g_2} \cap \dots \cap V_{\g_r}$ of elements from $\mathcal V$, thus forming a larger lists $\hat{\mathcal V}$ and a new corresponding list $\hat{\mathcal U}$ comprising all the intersections $U_{\g_1} \cap U_{\g_2} \cap \dots \cap U_{\g_r}$ . 

For each $k$, the locally-defined sets $\{V_{\g_l, k}\}_l$ have an intrinsic description in terms of the combinatorial patterns of tangency. So they automatically agree on multiple intersections: $X \cap V_{\g_l, k} \cap V_{\g_{m}} \subset  X \cap V_{\g_m, k}$ for all $l, m$. Now the same inductive argument in $k$ works for each map $$\Gamma: X \cap V_{\g_1} \cap V_{\g_2} \cap \dots \cap V_{\g_r} \to U_{\g_1} \cap U_{\g_2} \cap \dots \cap U_{\g_r}$$, so that this map is a weak homotopy equivalence as well.

As a result, by Corollary 1.4  \cite{May},  $\Gamma: X \to \mathcal T(v)$ is a week homotopy equivalence.





\bigskip

Now consider a traversing field $v$ on $X$. Let  $\mathcal A$ be any local coefficient system (a sheaf) on $X$ with an abelian group $\mathsf A$ for the stock.  We denote by  $\Gamma_\ast(\mathcal A)$ its push-forward residing on the trajectory space $\mathcal T(v)$. 

Let $U$ be an open neighborhood of a typical trajectory $\g$ in $X$.  Since $X$ is compact and a typical $\Gamma$-fiber---a trajectory $\g$---is closed, the canonical homomorphism $$\lim\; \textup{ind}_{\{U \supset \g\}}\; H^\ast(U; \mathcal A|_U) \to H^\ast(\g; \mathcal A|_\g)$$ is an isomorphism (see Theorem 4.11.1 in \cite{God}).  Since all $\g$'s are either segments or singletons, $H^\ast(\g; \mathcal A|_\g) = 0$ for all $\ast \neq 0$ and $H^0(\g; \mathcal A|_\g) = A$. Thus, the Leray spectral sequence $$\big\{E_2^{pq} = H^p\big(\mathcal T(v);\, \mathcal H^q(\g; \mathcal A)\big)\big\}_{p, q}$$  of the map $\Gamma: X \to \mathcal T(v)$ collapses (see Theorem 4.17.1 in \cite{God}). As a result, we get that the map $\Gamma$ establishes an isomorphism $\Gamma^\ast: H^\ast(\mathcal T(v); \Gamma_\ast(\mathcal A)) \approx H^\ast(X; \mathcal A)$. 

In particular, for a trivial local system $\mathsf A$, we get $H^\ast(X; \mathsf A) \approx H^\ast(\mathcal T(v); \mathsf A)$. 
\end{proof}

\noindent{\bf Remark 5.1.}
If a traversing field $v$ is such that $\mathcal T(v)$ has a homotopy type of a $CW$-complex, then by Whitehead Theorem \cite{Wh}, $\Gamma: X \to \mathcal T(v)$ is a homotopy equivalence.  In the next paper, we will prove that, for a traversally generic field $v$,  the trajectory space $\mathcal T(v)$ can be given the structure of a compact $CW$-complex. \hfill\qed

\bigskip

For any sub-poset $\Theta \subset \Omega^\bullet_{'\langle n]}$ and a traversally generic field $v$ on $X$, let us  consider the subsets $X(v, \Theta) \subset X$ and $\mathcal T(v, \Theta) \subset \mathcal T(v)$ comprised of points $x \in X$ or of trajectories $\g_x \in \mathcal T(v)$ whose divisors $D_{\g_x}$ have the combinatorial models prescribed by the elements of $\Theta$.

In particular, we will see that the webs of subspaces $$\{X(v, \omega_\succeq)\}_{\omega \in \Omega^\bullet_{'\langle n]}} \quad\text{and} \quad \{\mathcal T(v, \omega_\succeq)\}_{\omega \in \Omega^\bullet _{'\langle n]}}$$ form a remarkable geometric structure.  It will preoccupy us for the rest of this series of articles.

A cruder stratification (filtration) of $X$ and $\mathcal T(v)$ is provided by the spaces  
$$\{X(v, \Omega^\bullet_{'[k, n]})\}_{0 \leq k \leq n} \quad\text{and} \quad \{\mathcal T(v, \Omega^\bullet_{'[k, n]})\}_{0 \leq k \leq n}$$, respectively.
\smallskip


Lemma 3.4 (see also Lemma \ref{eq1.2}) and Theorem 3.5 from \cite{K2} have an useful implication: 

\begin{corollary}\label{cor5.1} Let $X$ be a  smooth $(n+1)$-manifold with boundary. For any traversally generic  field $v$, the obvious map $\Gamma : \d_1X \to \mathcal T(v)$ is $(n + 2)$-to-$1$ at most. At the same time, $\Gamma: \d_2X \to \mathcal T(v)$ is $n$-to-$1$ at most. For each $\omega$, the restriction of $\Gamma$ to the subspace $\Gamma^{-1}(\mathcal T(v, \omega))$ is $|\sup(\omega)|$-to-$1$. \smallskip

When restricted to the $\Gamma$-preimage of the proper stratum $\mathcal T(v, \omega)$, $\Gamma$ is a covering  map with a trivial monodromy and a fiber of cardinality $|\sup(\omega)|$.
\end{corollary} 

\begin{proof} For a traversally generic field, by Theorem  3.5 from \cite{K2}, $m(\g) \leq 2(n+1)$. Each trajectory has exactly two points of odd multiplicity, the rest of the points are tangent points of even multiplicity. Their  number does not exceed $n$. Thus $\Gamma$ is $(n + 2)$-to-$1$ at most, and  $\Gamma|_{\d_2X}$ is  $n$-to-$1$ at most. 

The  statement dealing with the cardinalities of the fibers of $$\Gamma: \Gamma^{-1}(\mathcal T(v, \omega)) \to \mathcal T(v, \omega)$$  follows instantly from the definitions of the relevant spaces. 

Let $\b$ be a loop in $\mathcal T(v, \omega)$, and let $E_\b := \Gamma^{-1}(\b) \subset X$. Note that $\Gamma: E_\b \to \b$, thanks to the orientation by $v$, is a cylinder (and not a M\"{o}bius band). Consider the intersection $E_\b \cap \d_1X$. Since $\b$ is contained in the pure stratum $\mathcal T(v, \omega)$,  $\Gamma: E_\b \cap \d_1X \to \b$ is a covering map with a finite fiber.  Because its space $E_\b \cap \d_1X$ is contained in the cylinder $E_\b$, we conclude that $\Gamma: E_\b \cap \d_1X \to \b$ must be  a trivial covering. 
\end{proof}

Assuming that $v$ is traversally generic, our immediate  goal is to describe one particular \emph{localized} cellular structure of the trajectory space $\mathcal T(v)$.  As we mentioned before, it is governed by the combinatorics of the divisors in $\R$ of real degree $\leq 2(n+1)$ polynomials.  

First, we would like to understand better the $\Omega^\bullet_{'\langle n]}$-stratified structure of  $\mathcal T(v)$, localized to the vicinity of given $v$-trajectory $\g$. 

As usual, we operate  within an  extension germ $(\hat X, \hat v)$ of $(X, v)$. Let $\{a_i\}_i : = \g \cap \d_1X$ be the the $v$-ordered finite set of points, where  $a_i \in \d_{j_i}X^\circ$. By  by Theorem  3.5 from \cite{K2} (see also Lemma \ref{lem1.2}), this  tangency pattern $\omega = (j_1, j_2, \dots )$ is described by an element $\omega \in \Omega^\bullet_{'\langle n]} \cap \Omega_{\langle 2n +2]}$.


By Lemma \ref{eq1.2} and formula  (\ref{eq1.7}), in special  coordinates $(u, x, y)$ on some $\hat v$-adjusted tube surrounding $\g$, the manifold $X$ is given by the inequality 
\begin{eqnarray}\label{eq5.1}
P(u, x):= \prod_i \;\big[(u - \a_i)^{j_i} + \sum_{l=0}^{j_i - 2} x_{i, l} (u - \a_i)^l\big]  \leq 0
\end{eqnarray}
, where $\a_i = u(a_i)$ and $x = \{x_{i, l}\}$.

Next, in our local analysis, we can pick a canonical model of $X$ in the vicinity of $\g$ by assuming that each $\a_i = i$. 

If $|\omega| \equiv 0 (2)$, for each fixed value of the coordinates $(x, y)$, the solution set of (\ref{eq5.1}) is a disjoint  union of several closed intervals and singletons residing in the $u$-line $\hat\g_x$ (see Fig. 1 in \cite{K2}), the union depending on $x$ alone. Each of these intervals and singletons represent a $v$-trajectory suspended over $(x, y)$ (for some $x$, $\hat\g_x$ can be empty!). So to get the space of trajectories $\mathcal T(v)$ in the vicinity of  $\g$, we need to collapse each interval to a point-marker  that resides in it. Let us formalize the collapsing procedure. \smallskip

Consider the solution set $E_\omega$ of (\ref{eq5.1}). We say that two points $(u, x, y), (u', x', y') \in E_\omega$  are equivalent (``$\sim$"), if $x = x', y = y'$, and the interval $([u, u'], x, y) \subset E_\omega$. Now we define the space $\mathsf T_\omega$ as the quotient space $E_\omega/\sim$. 

The space  $\mathsf T_\omega$ comes equipped with the map  $p: \mathsf T_\omega \to \R^{|\omega|'} \times \R^{n - |\omega|'}$ induced by the obvious projection $(u, x, y) \to (x, y)$. Since $X$ is compact, for each $x$, the polynomial  $P(u, x)$ in (\ref{eq5.1}) has finitely many intervals where it is negative, $p$ is a ramified map with finite fibers.

For any fixed $x$,  the $u$-polynomial in $P(u, x)$ in (\ref{eq5.1}) can be viewed also as an element of  the space $\mathcal P^{|\omega|} := \R^{|\omega|}_{\mathsf{coef}}$, and as such belongs to a unique pure stratum  $\mathsf R_{\omega'} := \mathsf R_{\omega'}(x) \subset \mathcal P^{|\omega|}$, where $\omega' \succeq \omega$ in $\Omega_{\langle |\omega|]}$. Therefore, with the help of  (\ref{eq5.1}), each $x \in \R^{|\omega|'}$ has a well-defined combinatorial type $\omega(x) = \omega' \in \Omega_{\langle |\omega|]}$ associated to it. As a result,  $\R^{|\omega|'}$ is an $\Omega_{\langle |\omega|]}$-stratified space, and so is $E_\omega \subset \R \times \R^{|\omega|'} \times \R^{n - |\omega|'}$.   

In fact, the space $E_\omega$ admit a ``more intrinsic" stratification which is labeled by the elements of the poset $\Omega^\bullet_{'\langle d]}$, where $d = |\omega|'$.  In a sense, this stratification is cruder than the  $\Omega_{\langle |\omega|]}$-stratification of $\R^{|\omega|'}$. Here is the description of this $\Omega^\bullet_{'\langle  |\omega|']}$-stratification. 

For each point $(u_\star, x, y) \in E_\omega$, there is a unique closed interval $I_{u_\star, x} := [a, b]$ such that $u_\star \in [a, b]$, $P(u, x) \leq 0$ for all $u \in [a, b]$, and $I_{u_\star, x}$ is the maximal closed interval possessing  these two properties. Consider the real zero divisor $D_{(u_\star, x)}$ of the $u$-polynomial $P(u, x)$ being restricted to the interval $I_{u_\star, x}$.  Its combinatorial type $\omega(u_\star, x) \in \Omega^\bullet$ and is independent on the choice of $u_\star$ within the interval $I_{u_\star, x}$. Thus, $\omega(u_\star, x)$ depends only on the equivalence class of $(u_\star, x, y) \in E_\omega$, a point  in $\mathsf T_\omega$. 

On the other hand, if  $\omega(x) = \omega(x')$ for some $x, x' \in \R^{|\omega|'}$,  then there exist $$(u_\star, x, y), (u'_\star, x', y) \in E_\omega$$ such that $\omega(u_\star, x) = \omega(u'_\star, x')$. Moreover, the orders which the intervals $I_{u_\star, x}$ and $I_{u'_\star, x'}$ occupy inside the sets $P^{-1}((-\infty , 0), x)$ and $P^{-1}((-\infty , 0), x')$, respectively, are the same. Stated differently, the combinatorial type $\omega(x) \in \Omega_{\langle |\omega|]}$ determines the ordered sequence $\Xi(\omega(x))$ of types $\{\omega(u_\star, x)\}_{u_\star}$ for points in the fiber $p^{-1}(x)$ (cf. the discussion preceding Fig. 1).  

Since the construction of the space $\mathsf T_\omega$ and its $\Omega^\bullet_{'\langle |\omega|']}$-stratification depends only on the combinatorial pattern $\omega \in \Omega^\bullet_{'\langle n]}$, we get:

\begin{theorem}\label{th5.2} For any traversally generic field $v$ on a $(n + 1)$-manifold $X$ and any $v$-trajectory $\g$ with the tangency multiplicity  pattern $\omega \in  \Omega^\bullet_{'\langle n]}$, the $\Omega^\bullet_{'\langle |\omega|']}$-stratified topological type of the germ of  the trajectory space $\mathcal T(v)$ at the point $\g$ is determined by the combinatorial pattern $\omega$ alone. \qquad \qquad \qed
\end{theorem}

\noindent{\bf Example 5.1.} For the traversally generic  fields $v$ on $4$-folds $X$,  there are $11$ distinct local topological  types for $\mathcal T(v)$. They are labeled by the elements of the poset from  Fig. 2. \hfill\qed
\smallskip
\smallskip

Next, we will employ similar considerations to describe the germ at $\g$ of a cellular structure in $\mathsf T_\omega$, a structure subordinate to the filtration of $\mathsf T_\omega$ by spaces which are labeled by the elements  of the poset $\Omega^\bullet_{'\langle |\omega|']}$. Eventually, these investigations will culminate in Theorem \ref{th5.3} below. 
\smallskip

By choosing an appropriately narrow $\hat v$-adjusted neighborhood $U \subset \hat X$ of $\g$, for any $x$ sufficiently close to the origin, the complex zeros of $P(u, x)$ from (\ref{eq5.1}) can be separated into disjointed groups.  These groups correspond to the real zeros $\{\a_i\}_i$ of the polynomial $P(u, 0)$ (which is a product of linear polynomials over $\R$) and reside in their vicinity. Moreover, each  portion of the complex zero divisor of $P(u, x)$, taken within each group, by (\ref{eq5.1}), is $\a_i$-balanced.  

Thus, for any $\hat v$-trajectory $\hat \g \subset U$,  the ``real" zero divisor $D_{\hat\g}$ splits into several 
``real" divisors $D_{\hat\g, i}$.  Their combinatorial types are described by some elements  $\hat\omega_i \in \Omega_{\langle \omega(i)]}$ such that   $|\hat\omega_i| \leq \omega(i)$ and $|\hat\omega_i| \equiv \omega(i)\, \mod (2)$.  In other words, $\omega_{\hat\g}$, the combinatorial type of $D_{\hat\g}$, is determined by an element of the product $\prod_i \, \Omega_{\langle  \omega(i)]}$. Evidently, there is  a canonical map $$\kappa: \prod_i \, \Omega_{\langle  \omega(i)]} \to \Omega_{\langle |\omega|]}$$ that places $\{\hat\omega_i\}_i$ in a single array $\omega_{\hat\g} := (\hat\omega_1, \hat\omega_2, \dots )$.


For each $i$, we denote by  $x_i$ the ordered subset $\{x_{i,1}, x_{i,2}, \dots \}$ of the $x$-coordinates, amenable to the formula (\ref{eq5.1}). 

The decomposition of $P(u, x) = \prod_i P_i(u - \a_i, x_i)$ of the left-hand side of (\ref{eq5.1})  into a product of monic depressed polynomials $P_i(u - \a_i, x_i)$ in $u -\a_i$ of degrees $j_i :=\omega(i)$ gives rise a continuous map
\begin{eqnarray}\label{eq5.2} 
A_\omega: \R^{|\omega|'} \to \prod_i \R_{\mathsf{coef}, \a_i}^{\omega(i) - 1}
\end{eqnarray} 
from the space of $x$-coordinates $\R^{|\omega|'}$, to the product of polynomial spaces $\{\R_{\mathsf{coef}, \a_i}^{\omega(i) - 1}\}_i$ of equal dimension $|\omega|'$. If $x \neq x'$, then at least for one $i$, the vectors  $x_i$, $x'_i$ are distinct. Thus the $(u - \a_i)$-polynomials $P_i(u - \a_i, x_i)$ and $P_i(u - \a_i, x'_i)$ must have distinct coefficients. As a result, $A_\omega$ is a 1-to-1 map. Evidently, by (\ref{eq5.1}), $A_\omega$ is surjective. So  $A_\omega$ establishes a smooth homeomorphism of the two spaces from (\ref{eq5.2}). 

With the help of Theorems \ref{th4.1} and \ref{th4.2}, we get the cellular structures
$$ 
\{\Phi_{\a_i, \hat\omega_i}: e_{\a_i, \hat\omega_i} \to \mathsf R_{\a_i, \hat\omega_i}\}_{\hat\omega_i \in \Omega_{\langle \omega(i)]}}
$$ 
for each space $\R_{\mathsf{coef}, \a_i}^{\omega(i) - 1}$. By $A_\omega^{-1}$ from (\ref{eq5.2}), this product of cellular structures in the target space of  (\ref{eq5.2}) gives rise to a cellular structure in the space $\R^{|\omega|'}$  of $x$-coordinates. Employing $$\kappa: \prod_i \, \Omega_{\langle  \omega(i)]} \to \Omega_{\langle |\omega|]}$$, that structure 
\begin{eqnarray}\label{eq5.3} 
\Psi_{\hat\omega} := A_\omega^{-1}\circ \big(\prod_i \Phi_{\a_i, \hat\omega_i}\big): \quad \prod_i  e_{\a_i, \hat\omega_i} \longrightarrow \R^{|\omega|'}
\end{eqnarray}
is consistent with the stratification of $\R^{|\omega|'}$, labeled by the elements $\kappa(\{\hat\omega_i\}_i)$ of the poset $\Omega_{\langle |\omega|]}$. 

In what follows, the cellular structure in the product $\R^{|\omega |'} \times \R^{n -|\omega|'}$ is chosen to be this cellular structure in $\R^{|\omega |'}$ given by  (\ref{eq5.3}), being multiplied by a single open cell $e^{n -|\omega|'} \approx \R^{n -|\omega|'}$. 

Note that, so far, the cells in $\R^{|\omega |'}$ are labelled by the elements of the poset $\omega_\preceq \subset  \Omega_{\langle |\omega|]}$, and not by elements of the poset $\omega_{\bullet\preceq} \subset  \Omega^\bullet_{'\langle |\omega|']}$, appropriate for the trajectories in $E_\omega$ (see the discussion preceding Lemma \ref{lem4.3}).

Next, using the ramified map $p: \mathsf T_\omega \to \R^{|\omega|'} \times \R^{n- |\omega|'}$ with finite fibers, we will employ the  cellular structure (\ref{eq5.3}) in the target space $\R^{|\omega|'}$ to produce a cellular  structure in the source space $\mathsf T_\omega$, so that $p$ will become a cellular map. That  task will preoccupy us for a while...  
\smallskip

With this goal in mind, we introduce  \emph{markers}, a new combinatorial contraption (see Fig. 1 and Fig. 4). 

For each $\omega \in \Omega_{\langle d]}$, consider the auxiliary $u$-polynomial $$\wp_\omega(u) = \prod_{\{i\, \in\, \sup(\omega)\}}  (u - i)^{\omega(i)}$$  of degree that  does not exceed $d$ and shares the same parity with it.



For a given $\omega$, we consider a pair $(\omega, k)$, where the \emph{marker} $k \in \sup(\omega)\subset \N$ is such that:

\begin{itemize}
\item either $\omega(k) \equiv 0 \mod (2)$ and $\wp_\omega(k - 0.5) > 0, \;\wp_\omega(k + 0.5) > 0$, 
\item or $\omega(k) \equiv 1 \mod (2)$ and $\wp_\omega(k - 0.5) > 0,\;  \wp_\omega(k + 0.5) < 0$.
\end{itemize}

We denote by $\Omega^\mu_{\langle d]}$ the set of all marked pairs $(\omega, k)$ as above.

Let us denote by $\Upsilon(\omega) \subset \N$ the set of markers $k$ associated with $\omega$. Each marker $k \in \Upsilon(\omega)$, as an element of linearly ordered set  $\Upsilon(\omega)$,  acquires its ordinal $p$.  We will denote the $p$-th marker in $\Upsilon(\omega)$ by  $k_p$. Let  $\Upsilon_p(\omega) \subset \N$ be the maximal set of consecutive natural numbers $j \geq k_p$ such that $\wp_\omega(u) \leq 0$ for all $u \in [k_p, j]$.
\smallskip


It is possible to extend the elementary operations $\mathsf M_j$ and $\mathsf I_j$ (merge and insert),  introduced in (\ref{eq2.1})-(\ref{eq2.3}) for the poset $\Omega_{\langle d]}$,  to the elements of the new set  $\Omega^\mu_{\langle d]}$. 

The new merge operation $\mathsf M_j^\mu: \Omega^\mu_{\langle d]} \to \Omega^\mu_{\langle d]}$ is  shown in Fig. 4.
Let 
\begin{eqnarray}\label{eq5.4}
\mathsf M_j^\mu(\omega, k) := (\mathsf M_j(\omega), \mu_j(k))
\end{eqnarray}
, where $\mu_j(k) \in \Upsilon(\mathsf M_j(\omega))$ is defined as follows.

If both $j$ and  $j + 1$ belong to the same set $\Upsilon_p(\omega)$, and the marker $k \in  \Upsilon_q(\omega), q \neq p$, then $\mu_j(k) = k$ as elements of the set $\Upsilon_q(\mathsf M_j(\omega)) = \Upsilon_q(\omega)$. If  $j$ and  $j + 1$ belong to  $\Upsilon_p(\omega)$, and the marker $k \in  \Upsilon_p(\omega)$, then again $\mu_j(k) = k$ as elements of the set $\Upsilon_p(\mathsf M_j(\omega)) \subset \Upsilon_p(\omega)$. 
 

At the same time, if  $j \in  \Upsilon_p(\omega)$,  $j + 1 \in  \Upsilon_{p + 1}(\omega)$,  and $k \in \Upsilon_{p + 1}(\omega)$, then $\mu_j(k)$ is the unique minimal element in $\Upsilon_p(\mathsf M_j(\omega))$. When $k \in  \Upsilon_q(\omega)$ and $q \neq p+1$, the marker keeps its minimal position within the  subset $\Upsilon_q(\omega)$.


\begin{figure}[ht]
\centerline{\includegraphics[height=3in,width=3.7in]{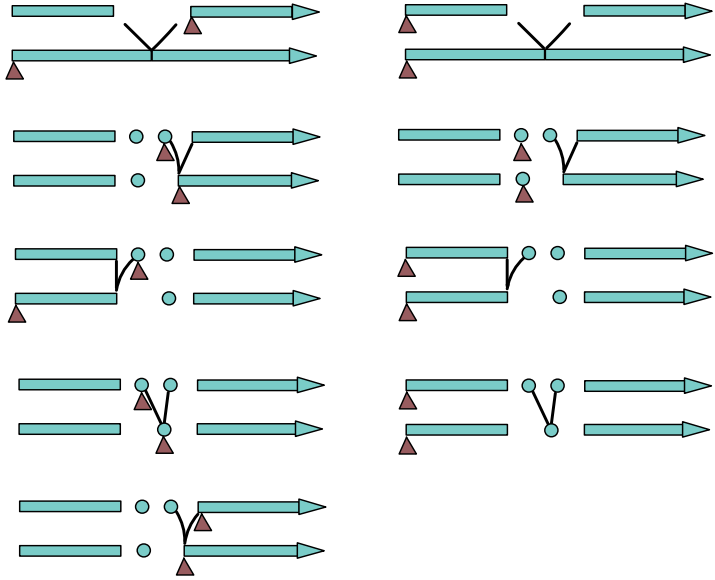}}
\bigskip
\caption{\small{The rules by which the markers evolve under the merge operations (depicted as the $V$-shaped passages from top to bottom bars in each of the seven diagrams).}}
\end{figure}

Similarly, the new insert operation 
\begin{eqnarray}\label{eq5.5}
\mathsf I_j^\mu(\omega, k) := (\mathsf I_j(\omega), \mu_j(k))
\end{eqnarray}
is described as follows. 
Under $\mathsf I_j^\mu$, $\omega$ is subjected to the old insert operation $\mathsf I_j$, while the marker $k \in \Upsilon_p(\omega)$ keeps its minimal position within the set $\Upsilon_p(\omega) \subseteq \Upsilon_q(\mathsf I_j(\omega))$ (where $q$ is uniquely determined by $p$ and $j$), that is, $\mu_j(k) = k$ as the minimal elements of the appropriate sets. 
\smallskip






Here is a slightly more geometrical look at the markers and their behavior, a look motivated by interactions of vector flows with the boundary $\d_1X$. Let $\hat\g \subset \hat X$ be a typical $\hat v$-trajectory in the vicinity of a given trajectory $\g \subset X$. The intersection  $\hat \g \cap \d_1X = \{b_l\}_{1 \leq l \leq q}$, allows us to \emph{shade}  some of the intervals $(b_l, b_{l+1})$ in which $\d_1X$ divides $\hat\g$: the interval is shaded if it belongs to $X$. Using any auxiliary function 
$z: \hat X \to \R$ (as in Lemma 3.1 from \cite{K2}), the shading is defined as the locus where $z|_{\g} \leq 0$.\footnote{Recall that $z$ is chosen to possess the following properties: 1) $0$ is a regular $z$-value  and $z^{-1}(0) = \d_1X$, 2) $z^{-1}((-\infty, 0]) = X$.}  
Therefore, the ordered sequence of the multiplicities $\omega_{\hat\g} := (m(b_1), m(b_2), \dots , m(b_q))$ uniquely determines which intervals $(b_l, b_{l +1})$ along $\hat\g$ are shaded (cf. Fig. 1, where the shading corresponds to "strings" and "atoms"): each shaded interval can be marked with a unique lowest \emph{odd}-multiplicity point in it; also each atom can be marked. \smallskip

As in Section 2 (see the discussion preceding Fig. 1), we denote by $\sup_{odd}(\omega)$ the  points $l \in \N$ in the support of $\omega$ such that $\omega(l) \equiv 1\, \mod (2)$ and by $\sup_{ev}(\omega)$ the  points $l$ in the support of $\omega$ such that $\omega(l) \equiv 0\, \mod (2)$.  We can count the elements of $\sup_{odd}(\omega)$ and pick only the ones that acquire odd numerals in that count (in this way, we pick half of the elements in $\sup_{odd}(\omega)$) (cf. Fig. 1). We denote this set by $\sup_{odd}^+(\omega)$ and its complement by $\sup_{odd}^-(\omega)$. We divide the points of $\sup_{ev}(\omega)$ also into two complementary sets: the first one, $\sup_{ev}^-(\omega)$, contains points that are bounded from below by a point from $\sup_{odd}^+(\omega)$  and from above by a point from  $\sup_{odd}^-(\omega)$; the second one, $\sup_{ev}^+(\omega)$, contains points that are bounded from below by a point from  $\sup_{odd}^-(\omega)$  and from above by a point from  $\sup_{odd}^+(\omega)$. In fact,  $\Upsilon(\omega) = 
\sup_{odd}^+(\omega) \cup \sup_{ev}^+(\omega)$.\smallskip

Consider $\omega' \succeq \omega$, where $\omega', \omega \in \Omega_{\langle m]}$, and two elements-markers $k' \in \Upsilon(\omega')$ and $k \in \Upsilon(\omega)$. Our next task is to define a relation ``$k' \leadsto k$" between $k'$ and $k$. Recall that $\omega$ can be obtained from $\omega'$ by a sequence of elementary merges and multiplicity 2 inserts as described in (\ref{eq2.1})-(\ref{eq2.3}). Given a marker $k' \in \Upsilon(\omega')$, we will describe its evolution under these elementary transformations of $\omega'$.  An insertion of an even multiplicity point does not affect any marker below the insertion and shifts the location of the marker above the insertion by one. If two points from $\sup_{ev}^+(\omega')$ merge and one of them is marked, the maker is placed at the location where the merge took place. If a marked point from $\sup_{ev}^+(\omega')$ is merging with a point from $\sup_{odd}^+(\omega')$, the location of the merge is marked. On the other hand, if a marked point from $\sup_{ev}^+(\omega')$  is merging with a point from $\sup_{odd}^-(\omega')$, then the marker is placed at the first point of $\sup_{odd}^+(\omega')$ located below the merge---the marker ``travels down until it reaches a point of $\sup_{odd}^+(\omega')$". Finally, if a marked point from $\sup_{odd}^+(\omega')$ merges with the adjacent point from $\sup_{odd}^-(\omega')$, the marker is placed at the first point of $\sup_{odd}^+(\omega')$ located below the merge.  If one inserts a point of an even multiplicity, the marker  does not change its location.

Therefore, when  $\omega' \succeq \omega$, the marker $k'$ defines a unique marker $k \in \Upsilon(\omega)$. In such a case, we say that $k'$ \emph{collapses} to $k$ and write ``$k' \leadsto k$".  
\smallskip


There is a natural and already familiar map $\Xi : \Omega^\mu \to \Omega^\bullet$ which acts from the set of all marked pairs $(\omega, k)$ to the set $\Omega^\bullet$ of strings and atoms (see Fig. 1).  By definition, $\Xi$ takes each pair $(\omega, k)$, $k \in \Upsilon(\omega)$,  to the  restriction of $\omega$ to the unique maximal interval $[k, j(k)]$ of indices in $\sup(\omega)$ with the property $\wp_\omega(j) \leq 0$ for all $j \in[k, j(k)]$. Then, by a shift of indices, one reinterprets $\omega: [k, j(k)] \to \N$ as a map $\omega: [1, j(k) - k -1] \to \N$, an element of $\Omega^\bullet$.   In fact,  $\Xi : \Omega^\mu_{\langle 2n + 2]} \to \Omega^\bullet_{'\langle n]}$ for each $n$.




\smallskip

In the following theorem, we combine the acquired knowledge about  one particular cellular structure of the model spaces $\{\mathsf T_\omega\}$ with the local models for traversally generic vector fields (described in Lemma \ref{lem1.2}). This leads to a \emph{local} purely combinatorial description of trajectory spaces for traversally generic fields.  Regretfully, the formulation of the theorem is lengthy. 

\begin{theorem}\label{th5.3} Let $X$ be a compact smooth $(n+1)$-manifold with boundary. Let $v$ be a traversally generic vector field on $X$ and  $\g$ its trajectory of the combinatorial type $\omega$.  Then the following statements hold: 
\begin{itemize}

\item In the vicinity of $\g$, the  trajectory space $\mathcal T(v)$ has a structure of the model $|\omega|'$-dimensional finite cell complex $\mathsf T_\omega$ times  $\R^{n - |\omega|'}$.\smallskip 

\item Each open cell $\mathsf E_{\hat\omega, k} \subset \mathsf T_\omega$ (of dimension $|\omega|' - |\hat\omega|'$) is indexed by an element $\hat\omega =   \prod_i \hat\omega_i$ of the poset $\prod_i \Omega_{\langle \omega(i)]}$\footnote{where the  partial order in  $\prod_i \, \Omega_{\langle \omega(i)]}$ is the one of the posets product.} 
together with a marker $k \in \Upsilon(\kappa(\hat\omega))$. Every point in  $\mathsf E_{\hat\omega, k}$ belongs to the pure stratum $\mathcal T\big(v,\, \Xi(\kappa(\hat\omega), k)\big)$ of $\mathcal T(v)$, labeled by the element $\Xi(\kappa(\hat\omega), k)\big) \in \Omega^\bullet_{'\langle n]}$. \smallskip

\item  $\mathsf E_{\hat\omega', k'} \subset \mathsf E_{\hat\omega_\succeq, k}$ if and only if  the following two conditions are satisfied: 

$\mathbf{(1)}$ $\hat\omega \succeq \hat\omega'$, and 

$\mathbf{(2)}$ the markers $k$ and $k'$ satisfy the relation `` $k \leadsto k'$".\smallskip

\item Employing the attaching maps $\Psi_{\hat\omega}: e_{\hat\omega} := \prod_i e_{\hat\omega_i} \to \R^{|\omega|'}$ 
from (\ref{eq5.3}), the space $\mathsf  T_{\omega}$ can be assembled from the marked cells $\{(e_{\hat\omega}, k)\}_{\hat\omega, k}$  according to the rules described in Fig. 4. Each  closed cell $\mathsf E_{\hat\omega_\succeq,\, k}$ is the image of the cell $(e_{\hat\omega}, k)$ under the attaching map $\Psi_{\hat\omega,\, k}: \coprod_{\hat\omega, k}(e_{\hat\omega}, k) \to \mathsf T_{\omega}$  based on (\ref{eq5.3}). \smallskip

\item Each cell $\mathsf E_{\hat\omega_\succeq,\, k}$, where $\kappa(\hat\omega) \succ \omega$, topologically  is a positive cone  with the apex at $\g \in \mathcal T(v)$ over a compact cell $\mathsf S_{\hat\omega_\succeq,\, k}$.  These cells $\{\mathsf S_{\hat\omega_\succeq,\, k}\}$ form a link of the point $\g$ in $\mathsf T_{\omega}$. The rules that describe their incidence are similar to the incidence rules for $\{\mathsf E_{\hat\omega_\succeq,\, k}\}$'s  in $\mathsf T_{\omega}$. \smallskip

\item $\mathsf T_{\omega}$ is equipped with a  ramified cellular map $p: \mathsf T_{\omega} \to \prod_i \R^{\omega(i) -1}_{\mathsf{coef},\, i}$ whose  fibers are of the cardinality $|\omega|/2$ at most. The image $\mathsf E_{\hat\omega} := p(\mathsf E_{\hat\omega, k})$ is a product of cells $\{\mathsf E_{\hat\omega_i}\}_i$.  The cellular structure $\{\mathsf E_{\hat\omega_i}\}_{\hat\omega_i \in \Omega_{\langle\omega(i)]}}$ in each space $\R^{\omega(i) - 1}_{\mathsf{coef},\, i}$ is described in Theorem \ref{th4.1}  and Lemma \ref{lem4.3} in terms of the poset $(\Omega_{\langle \omega(i)]}, \succ)$. 

\end{itemize}
\end{theorem}

\begin{figure}[ht]
\centerline{\includegraphics[height=2in,width=2.5in]{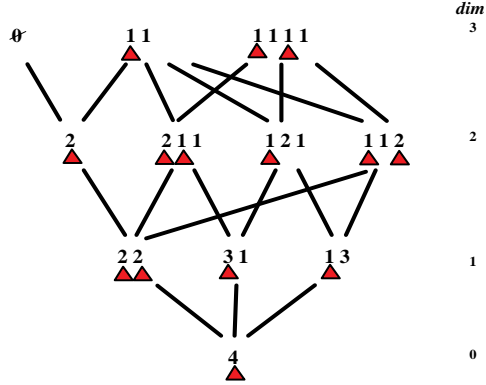}}
\bigskip
\caption{\small{The Swallow Tail singularity  poset $\Omega_{\langle 4]}$ with all possible markers $k$, denoted here by the symbol ``$_\triangle$". The dimensions of the strata are shown on the right. The diagram indicates that the $3$-complex $\mathsf T_{(4)}$ is assembled from three 3-cells, six 2-cells,  four 1-cells and one 0-cell.}}
\end{figure}

\begin{figure}[ht]
\centerline{\includegraphics[height=2in,width=2.5in]{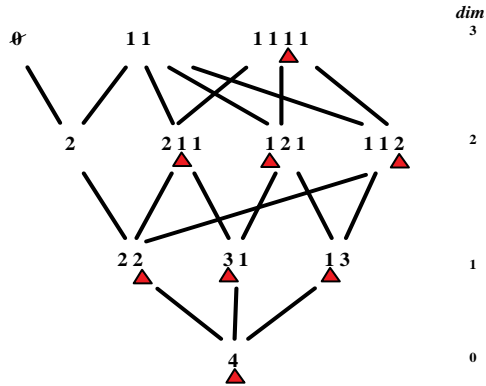}}
\bigskip
\caption{\small{The  poset $\Omega_{\langle 4]}$ with the downwards evolution ``$\leadsto$" of a particular marker $_\triangle$ attached to $\omega = 1 1 1 1$. The diagram indicates that the marked 3-cell $e_{1, 1, 1_\triangle, 1}$ in $\mathsf T_{(4)}$ shares its boundary with three marked 2-cells:  $e_{2, 1_\triangle, 1}$, $e_{1_\triangle, 2, 1}$, $e_{1, 1,  2_\triangle}$, and that each of these is attached to two marked 1-cells, and so on... }}
\end{figure}

\begin{proof} Recall that for $v \in \mathcal V^\dagger(X)$ and each trajectory $\hat\g \in \hat X$, the intersection $\hat\g \cap X$ consists of a number of closed ("shaded") intervals and few isolated points of even multiplicities. Each of these intervals or isolated points defines a singleton in the space $\mathcal T(v)$. We  place a single marker ``$_\Delta$" at one of the lower ends of the shaded intervals or at one of the isolated points, so that a trajectory $\hat\g \in \hat X$ and the marker $_\Delta \in \hat\g \cap X$ will determine a point in $\mathcal T(v)$. Thus a pair $(\hat\g,\, _\Delta)$, where $\hat\g \in \mathcal T(\hat v)$,  can be viewed a point of $\mathcal T(v)$. Here we distinguish between the combinatorial marker $k$ and its geometrical realization $_\Delta \in \d_1X \cap \hat\g$. 

In the vicinity of a given trajectory $\g \subset X$, each trajectory $\hat\g \subset \hat X$ is determined by the intersection point $\hat\g \cap S$, where $S$ denotes a local section of the $\hat v$-flow through a point on $\g$ which resides slightly below the first intersection $a_1 \in \g \cap \d_1X$.  By Theorems \ref{th4.1} and \ref{th4.2}, the combinatorial types $\hat\omega \in \prod_i \, \Omega_{\langle\omega(i)]}$, or rather their $\kappa$-images, label cells $\mathsf E_{\hat\omega}$ of codimension $|\kappa(\hat\omega)|'$ in $S$. To define the corresponding cell decomposition of $\mathcal T(v)$ in the vicinity of $\g$, we need multiple copies $\{\mathsf E_{\hat\omega, k}\}_k$ of each cell $\mathsf E_{\hat\omega}$, the copies that are indexed by all admissible markers $k$.  Their number is the cardinality of the set $\Upsilon(\kappa(\hat\omega))$. For each $\hat\g$,  there is a canonical correspondence between the markers $k \in \Upsilon(\kappa(\hat\omega))$ and the connected components $\hat\tau \in \pi_0(\hat\g \cap X)$. So we will  use elements $k \in \Upsilon(\kappa(\hat\omega))$ and $\hat\omega \in \prod_i \, \Omega_{\langle \omega(i)]}$ to label all the cells in $\mathcal T(v)$.
\smallskip

We are going to show that, for any $\hat\omega \succ \hat\omega'$ in $\prod_i \, \Omega_{\langle \omega(i)]}$, a $k$-marked cell $\mathsf E_{\hat\omega, k}$ is incident to a $k'$-marked cell $\mathsf E_{\hat\omega', k'}$, where $k \leadsto k'$, if and only if $\mathsf E_{\hat\omega}$ is incident to $\mathsf E_{\hat\omega'}$ in $S$. 

In the vicinity of $\g$, let us employ the special coordinates $(u, x, y)$ as in Lemma \ref{lem1.2}. Consider a smooth path $\{w(t)\}_{0 \leq t < 1}$ in the pure stratum $S(\hat\omega) \subset S$ so that $\lim_{t \to 1}w(t) = w_1 \in S(\hat\omega')$.  The  corresponding trajectories $\g_{w(t)} \subset \hat X$ converge to a trajectory $\g_{w_1}$ in such a way that  the sets $\g_{w(t)}\cap \d_1X$  go through merges of adjacent points or/and insertions of even multiplicity points. Examining the evolution of each  element $\hat\tau_t \in \pi_0(\g_{w(t)} \cap X)$ into an element  $\hat\tau'  \in \pi_0(\g_{w_1} \cap X)$, we see that it is captured by a marker collapse  $_{\Delta_t} \leadsto _{\Delta'}$ as depicted in Fig. 4.  Also glance at Fig. 6 which gives an example of such evolution. \smallskip

With the rigid rules for the collapses $\tau \leadsto \tau'$ from Fig. 4 in place, we can reconstruct the germ $\mathcal T_\g$ of  $\mathcal T(v)$ at $\g$ from the cell complex $S$ (whose cells are indexed by the elements of the poset  $(\prod_i \Omega_{\langle \omega(i)]}, \succ)$), and ultimately, from  $\omega = \omega(\g)$ itself. Speaking informally, in order to attach  a cell $\mathsf E_{\omega_\succeq,\, k}$ to a cell $\mathsf E_{\omega',\, k'}$, first we verify whether $k \leadsto k'$, and if the verification leads to a positive answer, we use the same map that attaches $\mathsf E_{\omega_\succeq}$ to $\mathsf E_{\omega'}$ in $S$.  

Let us describe a more formal construction of $\mathcal T_\g \approx \mathsf T_\omega \times \R^{n - |\omega|'}$.

For $\hat\omega = \{\hat\omega_i \in  \Omega_{\langle \omega(i)]} \}_i$, let $e_{\hat\omega}  = \prod_i e_{\hat\omega_i}$, where the hypersurface 
\begin{eqnarray}\label{eq5.6}
e_{\hat\omega_i}  \subset \mathsf{Sym}^{|\sup(\hat\omega_i)|}(\R) \times \mathsf{Sym}^{\frac{\omega(i) - |\hat\omega_i|}{2}}(\H)
\end{eqnarray}
consists of  divisors $(D'_i, D''_i)$ such that $D'_i + D''_i + \tau(D''_i)$ is $i$-balanced.

Then $\mathsf T_\omega$ is the quotient of the space
 
\begin{eqnarray}\label{eq5.7}
\mathsf Z_{\omega} := 
\coprod_{\hat\omega \in \prod_i \Omega_{\langle \omega(i)]},\, \; k \in \Upsilon(\kappa(\hat\omega))}
\; (e_{\hat\omega}, k)
\end{eqnarray} 
by the equivalence relation ``$(z, k) \sim (z', k')$"  that is defined as follows:
\begin{itemize}
\item $\hat\omega \succ \hat\omega'$ in the poset $\prod_i (\Omega_{\langle \omega(i)]}, \succ)$, 
\item $z \in \d e_{\hat\omega},\; z' \in e^\circ_{\hat\omega'}$ are such that $\Psi_{\hat\omega}(z) =  \Psi_{\hat\omega'}(z')$ in the space $\prod_i \R^{\omega(i) -1}_{\mathsf{coef},\, i}$, where the $\Psi$-maps are defined by (\ref{eq5.3}),
\item $k \leadsto k'$.
\end{itemize}

Each closed cell $\mathsf E_{\hat\omega_\succeq, k} \subset \mathsf T_\omega$ is defined as the equivalence class of the set $(e_{\hat\omega}, k)$  in the quotient space $\mathsf Z_{\omega}/ \sim$. Its interior is homeomorphic to  an open ball $e_{\hat\omega}^\circ$ of dimension $|\omega|' - |\hat\omega|'$. 

As a result, for a traversally generic $v$-flow, we have precise instructions for assembling the germ of a cell complex $\mathcal T_\g$ associated with each trajectory $\g \in \mathcal T(v)$, or rather, with its combinatorial type $\omega$.  

In the vicinity of $\g$,  the  projection $p: \mathcal T(v) \to S$ is evidently a finitely-ramified cellular map 
with respect to the cellular structures $\{\mathsf E_{\hat\omega, k}\}_{\hat\omega, k}$ and $\{\mathsf E_{\hat\omega}\}_\omega$.  

Note that not every cell $\mathsf E_{\hat\omega} \subset S$ is the $p$-image of some cell from $\mathcal T(v)$: the cells in $S$ that are pierced by the trajectories  $\hat\g \subset \hat X$ with the property $\hat\g \cap X = \emptyset$ are not in  $p(\mathcal T(v))$. However, each $\mathsf E_{\hat\omega}$ with $\sup(\hat\omega) \neq \emptyset$ belongs to $p(\mathcal T(v))$. 
\smallskip

For each $\hat\omega$, the $p$-fiber over the pure stratum  $S(\kappa(\hat\omega))^\circ \subset S$ is of the cardinality $\#[\Upsilon(\kappa(\hat\omega))]$. We would like to get an upper bound for $\#[\Upsilon(\kappa(\hat\omega))]$ in terms of the $\omega$. In fact, $m(\g)/2 = |\omega|/2$ is such an upper bound. Indeed, the divisor $D_\g$ can be resolved in a divisor with $m(\g)$ simple roots. Recall that $m(\g)$ must be even, and, by Theorem 3.3  and Theorem  3.5 from \cite{K2}, any potential resolution of $D_{\g}$ (with the combinatorics described by elements of $\kappa(\prod_i \Omega_{\langle \omega(i)]})$) is realized in the vicinity of $\g$. Therefore $m(\g)/2$ shaded intervals do occur in the vicinity of $\g$. In fact, for traversally generic $v$, the number $m(\g)/2 \leq \dim(X)$ is the maximal possible cardinality of the $p$-fibers in the vicinity of $\g$.
\smallskip

Next we turn to describing the positive cone structure of $\mathcal T(v)$ at $\g$, which is consistent with its cellular structure $\{\mathsf E_{\hat\omega, k}\}_{\hat\omega, k}$.  In order to construct a base $\s_{\hat\omega}$  for the cone structure  in each cell $e_{\hat\omega}  = \prod_i e_{\hat\omega_i}$ (see (\ref{eq4.5})  and (\ref{eq5.3})), consider a set of disjoint closed $2$-balls $\{B(i) \subset \C\}_i$ of radius $1/3$ and with the center at $(i, 0) \in \C$. 
We denote by $\s_{\hat\omega} \subset e_{\hat\omega}$ the set of complex  conjugation-invariant divisor pairs $(D', D'') = \oplus_i (D'_i, D''_i)$, where the divisor $D'_i + D''_i + \tau(D''_i)$ of degree $\omega(i)$ is $i$-balanced and  the $\sup(D'_i  + D''_i) \subset B(i)$; furthermore, at least for one $i$,  $\sup(D'_i + D''_i)$ is not contained in the interior of  $B(i)$. The group $\R^\ast_+$ acts semi-freely  on $e_{\hat\omega} = \prod_i e_{\hat\omega_i}$ by the diagonal action that  applies a $t$-dilatation centered on the point $(i, 0) \in \C$ to each  pair $(D'_i, D''_i) \in e_{\hat\omega_i}$. Evidently, for each pair $(D', D'')$ there is a single $t \in \R^\ast_+$ so that $t(D', D'') \in \s_{\hat\omega}$. This $\R_+^\ast$-action extends to the space $\mathsf Z_{\omega}$ in (\ref{eq5.3}), the action on the markers being trivial. 

By Theorem \ref{th4.1} and with the help of the Vi\`{e}te map $V$,  a similar semi-free action is available on the space $\prod_i \R^{\omega(i) - 1}_{\mathsf{coef}, i}$. Since  the maps $\prod_i \Phi_{i, \hat\omega_i}$ are $\R^\ast_+$-equivariant, the quotient space  $\mathsf  T_\omega  = \mathsf  Z_{\omega}/\sim$ inherits  a semi-free $\R_+^\ast$-action for which $\g$ is the only fixed point.  

Let $\mathsf S_{\hat\omega, k}$ be the image of $\s_{\hat\omega} \times k$ under the obvious map $\mathsf  Z_{\omega} \to  \mathsf T_{\omega}$. 
Any nontrivial $\R_+^\ast$-trajectory meets $\mathsf S_{\hat\omega, k}$ at a singleton. Therefore $\mathsf E_{\hat\omega, k}$ is a positive cone over $\mathsf S_{\hat\omega, k}$. By their constructions, these cone structures in the individual cells $\mathsf E_{\hat\omega, k}$ are well-correlated and produce a positive cone structure in $\mathsf T_\omega$. As a result, $\{\mathsf S_{\hat\omega, k}\}_{\hat\omega, k}$ define a cellular structure in the link of $\g$ in $\mathcal T_\g = \mathsf T_\omega \times \R^{n - |\omega|'}$.
\end{proof}

Here is a short summary of what we have established in this paper: a traversally generic $v$-flow generates  a stratification  $\{\mathcal T(v, \tilde\omega)\}_{\tilde\omega \in \Omega_{\bullet '\langle n]}}$ of the trajectory space $\mathcal T(v)$, which is consistent  (in a  subtle way!) with, but cruder than, the $\g$-local cellular structure 
$$\big\{\mathsf E_{\hat\omega, k} \times \R^{n - |\omega|'}\big\}_{\hat\omega \in \prod_i \Omega_{\langle\omega(i)]},\;\; k \in \Upsilon(\kappa(\omega))}$$ 
in the model space $\mathcal T_\g \approx \mathsf T_{\omega} \times \R^{n - |\omega|'}$ that we just have described. Each cell $\mathsf E_{\hat\omega, k} \times \R^{n - |\omega|'}$ belongs to the stratum $\mathcal T(v, \tilde\omega)$, where $\tilde\omega := \Xi(\kappa(\hat\omega)) \in \Omega_{\bullet '\langle n]}$. In the process, distinct cells could acquire the same label $\tilde\omega \in \Omega_{\bullet '\langle n]}$.
\bigskip

This understanding of the cellular structure of the spaces of real polynomials of the degree $2n+2$---the structure which reflect the universal posets $\Omega_{\langle 2n+2]}$ and  $\Omega_{\bullet '\langle n]}$---will enable us, in the papers to follow, to define new rich characteristic classes of traversally generic  flows on $(n+1)$-manifolds with boundary.

\end{document}